\RequirePackage{fix-cm}
\documentclass[letterpaper]{amsart}

\usepackage{amsmath}
\usepackage{amsfonts}
\usepackage{hyperref}

\usepackage[x11names,rgb]{xcolor}
\usepackage{tikz}
\usetikzlibrary{snakes}
\usetikzlibrary{arrows}
\usetikzlibrary{shapes}
\usetikzlibrary{backgrounds}
\usetikzlibrary{matrix}
\usetikzlibrary{arrows}

\newcommand{\ap}{\mbox{\scriptsize ap}}
\newcommand{\supp}{\operatorname{Supp}} 
 
\newcommand{\asso}{\operatorname{Asso}}

\begin{document}

\newtheorem{definition}{Definition}
\newtheorem{lemma}{Lemma}
\newtheorem{proposition}{Proposition}
\newtheorem{claim}{Claim}
\newtheorem{theorem}{Theorem}
\newtheorem{conjecture}{Conjecture}
\newtheorem{corollary}{Corollary}
\theoremstyle{remark}
\newtheorem{remark}{Remark}

\numberwithin{definition}{section}
\numberwithin{lemma}{section}
\numberwithin{proposition}{section}
\numberwithin{claim}{section}
\numberwithin{theorem}{section}
\numberwithin{conjecture}{section}
\numberwithin{corollary}{section}
\numberwithin{remark}{section}
\numberwithin{equation}{section}

\title[Polyhedral models for generalized associahedra]{Polyhedral models for
	generalized associahedra via Coxeter elements}
\author{Salvatore Stella}
\address{\noindent Department of Mathematics, Northeastern University, Boston,
	MA 02115}
\email{stella.sa@husky.neu.edu}
\thanks{The author is partially supported by A. Zelevinsky's NSF grants
	DMS-0801187 and DMS-1103813}
\maketitle

Motivated by the theory of cluster algebras, F. Chapoton, S. Fomin and A. Zele\-vinsky 
associated to each finite type root system a simple convex polytope called
\emph{generalized associahedron}.
They provided an explicit realization of this polytope associated with a
bipartite orientation of the corresponding Dynkin diagram.

In the first part of this paper, using the parametrization of cluster variables
by their $g$-vectors explicitly computed by S.-W. Yang and A. Zelevinsky, we
generalize the original construction to any orientation. 
In the second part we show that our construction agrees with the one given by C.
Hohlweg, C. Lange, and H. Thomas in the setup of Cambrian fans developed by N.
Reading and D. Speyer.

\section{Introduction}
Much information on the structure of a cluster algebra $\mathcal{A}$ can be
deduced directly from a purely combinatorial gadget: its \emph{cluster complex}.
It is an abstract simplicial complex whose vertices are the cluster variables of
$\mathcal{A}$ and whose maximal simplices are given by clusters.
In this paper we restrict our attention to finite type cluster algebras; under
such assumption the cluster complex is finite.  

A complete classification of finite type cluster algebras was given in
\cite{ca2}: it is identical to the Cartan-Killing classification of semisimple
Lie algebras and crystallographic root systems.  In the same paper, under the assumption
that the initial cluster is \emph{bipartite}, Fomin and Zelevinsky provided an
explicit combinatorial description of the cluster complex obtained by labeling
its vertices with \emph{almost-positive roots} in the corresponding root system. 

They constructed a function on ordered pairs of labels, called
\emph{compatibility degree}, encoding whether the corresponding cluster
variables are \emph{compatible} (i.e. they belong to the same cluster),
exchangeable or neither. Its definition is purely combinatorial and does not
refer to the cluster algebra but just to the labels.  The description of the
cluster complex they presented is in terms of this function: compatible pairs of
almost positive roots form its $1$-skeleton; higher dimension simplices are
given by the cliques of the $1$-skeleton.

In \cite{y-system} the authors improved on this combinatorial model explaining how
almost positive roots give a geometric realization of the cluster complex.
They showed that the positive real span of the labels in any simplex of the
cluster complex is a cone in a complete simplicial fan: the \emph{cluster fan}.
Among the applications of this realization there is a parametrization of
cluster monomials in $\mathcal{A}$ with points of the root lattice $Q$ and an
explicit formula for all the exchange relations in the coefficient-free case.

Further study (\cite{associahedra}) of the cluster fan showed that it is the
normal fan of a distinguished polytope: the \emph{generalized associahedron} of
the given type. Its description is completely explicit: the authors discussed
all the constrains that its \emph{support function} must satisfy and then
provided a concrete function that meets them.

As noted above the construction in \cite{y-system} and \cite{associahedra}
depends on the labeling of cluster variables of $\mathcal{A}$ by almost positive
roots; such a parametrization is provided by their \emph{denominator vectors}
with respect to a bipartite initial cluster. 
Using the notion of \emph{$g$-vectors} from \cite{ca4}, in \cite{shih} Yang and
Zelevinsky generalized this parametrization of cluster variables to a family of
parametrizations, one for each acyclic initial cluster, by a subset $\Pi(c)$ of
the associated weight lattice (as customary we use a \emph{Coxeter element} $c$
in the Weyl group to keep track of the orientation of the initial cluster; see
(\ref{eqn:orientation-1}) and (\ref{eqn:orientation-2}) for details on the
conventions we adopt).

The first goal of this paper is to extend the results
from \cite{y-system} and \cite{associahedra} to each of these new
parametrizations.  Retracing the steps in those papers, for any choice of
acyclic initial cluster, we will construct a complete simplicial fan realizing
the cluster complex and we will show that it is the normal fan to a geometric
realization of a generalized associahedron. We can summarize our
claims as follows:
\begin{theorem}
	Let $\mathcal{A}$ be a cluster algebra of finite type with an acyclic initial
	cluster and let $c$ be the Coxeter element encoding the initial orientation.
	Let $\Pi(c)$ be the labeling set and $(\bullet||\bullet)_c$ its compatibility
	degree function both constructed in \cite{shih}. Then
	\begin{enumerate}
		\item 
			every $c$-cluster in $\Pi(c)$ (i.e. every maximal subset of $\Pi(c)$
			consisting of pairwise compatible weights) is a $\mathbb{Z}$-basis of the
			weight lattice $P$.

		\item
			The positive linear spans of the simplices in the clique complex induced
			by $(\bullet||\bullet)_c$ on $\Pi(c)$ form a complete simplicial fan
			$\mathcal{F}_c^\Pi$ realizing the cluster complex. Cluster monomials of
			$\mathcal{A}$ are in bijection with points of $P$.
			
		\item
			$\mathcal{F}_c^\Pi$ is the normal fan to a simple polytope: a geometric
			realization of the associated generalized associahedron.

		\item
			If $\mathcal{A}$ is coefficient-free then all its exchange relations are
			explicitly determined by the labels of exchangeable cluster variables.
	\end{enumerate}
\end{theorem}
The proof will be split into sub-statements, namely Theorems
\ref{thm-pi:z-basis}, \ref{thm-pi:complete-fan}, \ref{thm-pi:polytopality}, and
\ref{thm-pi:exchange-relations}. Some of these results were already
proved in less generality or were already conjectured; we will provide explicit
references in Section \ref{sec:preliminaries}.

It turns out that our polytopes are the same as those studied in
\cite{sortable-polytopality} in the setup of \emph{Cambrian fans} developed by
Reading and Speyer. The construction we propose, however, is different from the
one by Hohlweg, Lange, and Thomas. This provides us with an alternative
prospective on $c$-cluster combinatorics that allows us to recover all the exchange
relations of the associated coefficient-free cluster algebra and to answer
positively to Problem 4.1 posed in \cite{review-associahedra}. 

To explain what we mean by ``different'' recall that the definition of Cambrian
fans is given in terms of its maximal cones as opposed to the definition of
cluster fans that builds up from the $1$-skeleton. Indeed to each Coxeter element
$c$ of a finite type Weyl group $W$ one can associate a lattice congruence on
the group itself (seen as a lattice for the right weak order). This produces a
coarsening of the associated \emph{Coxeter fan} obtained by glueing together
cones corresponding to elements in the same class (recall that the Coxeter fan
is the complete simplicial fan in the weight space of $W$ whose maximal cones
are the images of the fundamental Weyl chamber under the action of the group).
The approach used in \cite{sortable-polytopality} to show that the Cambrian
fans are polytopal follows the same philosophy: they begin from the
\emph{generalized permutahedron} associated to $W$ seen as intersection of
half-spaces and, again using the lattice congruence induced by $c$, they remove
a certain subset of them to make it into a generalized associahedron.

The second goal of this paper is to show that the generalizations of
the cluster fans we propose coincide with the Cambrian fans of Reading and
Speyer.  To do so it suffices to show that the polyhedral models for the
generalized associahedra  we build are the same as the realizations given in
\cite{sortable-polytopality}.
Note that, in type $A$,  the interaction between
the geometric realizations of the associahedron by Hohlweg, Lange, and Thomas
and the original realization by Chapoton, Fomin, and Zelevinsky has been already
investigated in \cite{many-realizations}.

The paper is structured as follows: in Section \ref{sec:preliminaries}, after
having recalled the required terminology and having set up some notations, we
discuss in more details our generalizations of the results in
\cite{y-system} and \cite{associahedra} and we provide an idea of the
strategy we adopt to prove them. We then recall some more terminology
and explain how our construction relates to Cambrian fans and to the polytopes
from \cite{sortable-polytopality}.

In section \ref{sec:phi-ap} we introduce the main tool of the paper: the set of
\emph{$c$-almost-positive roots} $\Phi_{\ap}(c)$.  Many arguments from
\cite{y-system} and \cite{associahedra} require to perform an induction on the
rank of the cluster algebra; the labeling of cluster variables by
almost-positive roots is ideal for such a purpose. In our case, however, we are
given a set of weights to parametrize the vertices of the cluster complex
therefore we can not generalize those proofs directly. The solution we adopt is
to identify the weight lattice with the root lattice in such a way that the
restriction to a smaller rank cluster sub-algebra can be expressed easily in
terms of the labels in a new set $\Phi_{\ap}(c)$ (the image of $\Pi(c)$ under
this identification).

Section \ref{sec:bipartite} deals with a bipartite orientations. We show that, in
this case, our results follow directly from their analogues from \cite{y-system}
and \cite{associahedra}.

Section \ref{sec:technical-results} contains the proofs of some technical
results we need in Section
\ref{sec:proof-main-results} where we complete the proofs of the main results of
the first part of the paper.

The paper is concluded by Section \ref{sec:cambrian} where we show
that our realizations of the generalized associahedra coincide with those
constructed by Hohlweg, Lange, and Thomas and therefore that our generalization
of cluster fans is a different presentation of Cambrian fans.

\section{Preliminaries} 
\label{sec:preliminaries}
We start by setting up notation and recalling some terminology and results from
\cite{shih}.
Let $I$ be a finite type Dynkin diagram;  with a small abuse of notation denote
by $I$ also its vertex set.  Let $W$ be the associated Weyl group with simple
reflections $\left\{ s_1\right\}_{i\in I}$ and let $A=(a_{ij})_{i,j\in I}$ be
the corresponding Cartan matrix. 

Recall that an element $c$ of $W$ is said to be \emph{Coxeter} if every simple
reflection appears in a  reduced expression of $c$ exactly once. To each Coxeter
element $c$ associate a skew-symmetrizable matrix $B(c)=(b_{ij})_{i,j\in I}$ as
follows. For $i$ and $j$ in $I$, write $i\prec_c j$ if $i$ and $j$ are connected
by an edge  and $s_i$ precedes $s_j$ in a reduced expression of $c$. Set then 
\begin{equation}
	b_{ij}:=\left\{ 
	\begin{array}[]{ll}
		-a_{ij}& \mbox{if } i\prec_c j\\
		a_{ij}& \mbox{if } j \prec_c i\\
		0 & \mbox{otherwise}\,.
	\end{array}\right.
	\label{eqn:orientation-1}
\end{equation}
Note also that Coxeter elements are in bijection with orientation of $I$
under the convention 
\begin{equation}
	j\rightarrow i
	\qquad\Leftrightarrow\qquad
	i\prec_c j
	\,.
	\label{eqn:orientation-2}
\end{equation}
\begin{remark}
	In each Weyl group there is a distinguished class of Coxeter elements (call
	them \emph{bipartite}) corresponding to orientations of $I$ in which each node
	is either a source or a sink. Following the notation of \cite{y-system}, we
	denote bipartite Coxeter elements by $t$.
\end{remark}

For a given Coxeter element $c$ denote by $\mathcal{A}_0(c)$ the
coefficient-free cluster algebra with the initial $B$-matrix $B(c)$.  Let
$\left\{ \omega_i \right\}_{i\in I}$ be the set of fundamental weights
associated to $I$ and $w_0$ the longest element in $W$.  Set $h(i;c)$ to be the
minimum positive integer such that 
\begin{equation*}
	c^{h(i;c)}\omega_i=-\omega_{i^*}
\end{equation*}
where $\omega_{i^*}:=-w_0\omega_i$ (Cf. Proposition 1.3 in \cite{shih}).

By theorem 1.4 in \cite{shih} the set of weights
\begin{equation*}
	\Pi(c):=\left\{ c^m\omega_i : i\in I, 0\le m\le h(i;c) \right\}
\end{equation*}
parametrize the cluster variables in $\mathcal{A}_0(c)$. The correspondence is
given associating to each cluster variable its $g$-vector as defined in
\cite{ca4}; in particular cluster variables in the initial cluster correspond to
fundamental weights. 

The set $\Pi(c)$ can be made into an abstract simplicial complex of pure
dimension $n-1$ (the \emph{$c$-cluster complex}) as follows. 
The cluster algebra structure induces a permutation on $\Pi(c)$
\begin{equation*}
	\tau_c^\Pi(\lambda):=
	\left\{ 
	\begin{array}[]{ll}
		\omega_i &  \mbox{if } \lambda=-\omega_i\\
  	c\lambda & \mbox{otherwise}
	\end{array}
	\right.
\end{equation*}
and a (unique) $\tau_c^\Pi$-invariant \emph{$c$-compatibility degree} function
defined by the initial conditions
\begin{equation*}
	\left( \omega_i||\lambda \right)^\Pi_c:=
  \left[ (c^{-1}-1)\lambda;\alpha_i \right]_+
\end{equation*}
where $[\bullet;\alpha_i]$ is the coefficient of $\alpha_i$ in $\bullet$
expressed in the basis of simple roots and $\left[ \bullet \right]_+$ denotes
$\max\left\{ \bullet,0 \right\}$ (Cf. Proposition 5.1 in \cite{shih}).

Note that the action of $\tau_c^\Pi$ on $\Pi(c)$ is, by construction, compatible
with the action of $w_0$ on $I$; that is any $\tau_c^\Pi$-orbit contains a
unique pair $\left\{ \omega_i,\omega_{i^*} \right\}$ (or a single fundamental
weight $\omega_i$ if $i=i^*$).

Call two weights $\lambda$ and $\mu$ in $\Pi(c)$ \emph{$c$-compatible} if
\begin{equation*}
	\left( \lambda||\mu \right)_c^\Pi=0
	\,.
\end{equation*}
This definition makes sense since the $c$-compatibility degree satisfies
\begin{equation*}
	\left( \lambda||\mu \right)_c^\Pi=0
	\Leftrightarrow
	\left( \mu||\lambda \right)_c^\Pi=0
	\,.
\end{equation*}

The $c$-cluster complex $\Delta_c^\Pi$ is defined to be the abstract simplicial
complex on the vertex set $\Pi(c)$ whose 1-skeleton is given by $c$-compatible
pairs of weights and whose higher dimensional simplex are given by the cliques
of its 1-skeleton.  We refer to its maximal simplices as \emph{$c$-clusters};
this name already appeared in the work of Reading and Speyer in a different
setup, we will discuss later on how the two notions are related.

The first step in order to construct a complete simplicial fan realizing the
$c$-cluster complex is to show that we can associate an $n$-dimensional cone to
each $c$-cluster.
\begin{theorem}
	Each $c$-cluster in $\Delta_c^\Pi$ is a $\mathbb{Z}$-basis of the weight
	lattice $P$. 	
	\label{thm-pi:z-basis}
\end{theorem}

\begin{remark}
	Theorem \ref{thm-pi:z-basis} was conjectured in \cite{ca4} (Conjecture
	7.10(2)) and then proved in \cite{dwz2} (Theorem 1.7) under the assumption
	that the initial exchange matrix is skew-symmetric.
\end{remark}

Let $\mathcal{F}_c^\Pi$ be the collection of all the cones in
$P_\mathbb{R}$ that are positive linear span of simplices in the $c$-cluster
complex.

\begin{theorem}
	$\mathcal{F}_c^\Pi$ is a complete simplicial fan.
	\label{thm-pi:complete-fan}
\end{theorem}

\begin{remark}
	This is a generalization of Theorem 1.10 in \cite{y-system}, and our proof is
	inspired by the one in that paper. In particular we will deduce the result
	from the following proposition (mimicking Theorem 3.11 in there).
\end{remark}

\begin{proposition}
	Every point $\mu$ in the weight lattice $P$ can be uniquely be written as
	\begin{equation}
		\mu=\sum_{\lambda\in\Pi(c)}m_\lambda\lambda
		\label{eqn:pi-cluster_expansion}
	\end{equation}
	where all the coefficients $m_\lambda$ are non-negative integers and
	$m_\lambda m_\nu=0$ whenever $\left( \lambda||\nu \right)_c^\Pi\neq0$
	\label{prop-pi:unique-cluster-expansion}
\end{proposition}
The expression (\ref{eqn:pi-cluster_expansion}) is called the \emph{$c$-cluster
expansion} of $\mu$.

A simplicial fan is said to be \emph{polytopal} if it is the normal fan to a
simple polytope. Recall that, given a simple full-dimensional polytope $T$ in a
vector space $V$ its \emph{support function} $F$ is the piecewise linear
function on $V^*$ defined by
\begin{equation*}
	\begin{array}[]{cccl}
		F:&V^*&\longrightarrow& \mathbb{R}\\
		&\varphi&\longmapsto&\displaystyle\max\left\{ \varphi(x) | x\in T \right\}
	\end{array}
\end{equation*}
and its \emph{normal fan} is the complete simplicial fan in $V^*$ whose maximal
cones are the domains of linearity of $F$. Note that, in dimension greater than
2, not every simplicial fan  needs to be the normal fan of a polytope (see for
example section 1.5 in \cite{fulton}).

Our next goal is to show that the $c$-cluster fans we constructed so far are
polytopal.
In view of Theorem \ref{thm-pi:complete-fan}, each function defined on $\Pi(c)$
extends uniquely to a continuous, piecewise linear function on $P_\mathbb{R}$
linear on the maximal cones of $\mathcal{F}_c^\Pi$. In particular, every
function 
$$
	f:I\longrightarrow \mathbb{R}
$$
satisfying $f(i)=f(i^*)$ gives rise to a continuous, $\tau_c^\Pi$-invariant,
piecewise-linear function $F_c=F_{c;f}$, by setting 
\begin{equation*}
	F_c(c^m\omega_i):=f(i)
\end{equation*}
for all $c^m\omega_i\in \Pi(c)$, and then extending it to $P_\mathbb{R}$ as
above.

Let $\asso_c^f(W)$ be the subset of $P_\mathbb{R}^*$ defined by
\begin{equation}
	\asso_c^f(W):=
	\left\{\varphi\in P_\mathbb{R}^*\,|\, \varphi(\lambda)\le F_c(\lambda), \, \forall\lambda\in\Pi(c)  \right\}
	\,.
\end{equation}

\begin{theorem}
	If $f:I\rightarrow\mathbb{R}$ is such that
	\begin{enumerate}
		\item 
			for any $i\in I$
			$$f(i)=f(i^*)$$
			\label{thm-pi:polytopality:1}
		\item
			for any $j\in I$
			$$\sum_{i\in I}a_{ij}f(i)>0$$
			\label{thm-pi:polytopality:2}
	\end{enumerate}
	then $\asso_c^f(W)$ is a simple $n$-dimensional polytope with support
	function $F_c$.  Furthermore, the domains of linearity of $F_c$ are exactly
	the maximal cones of $\mathcal{F}_c^\Pi$, hence the normal fan of
	$\asso_c^f(W)$ is $\mathcal{F}_c^\Pi$.
	\label{thm-pi:polytopality}
\end{theorem}

\begin{remark}
	Theorem \ref{thm-pi:polytopality} is a generalization of Theorem 1.5 in
	\cite{associahedra}. Its proof uses the result by Chapoton, Fomin, and
	Zelevinsky as base case.
\end{remark}

The following examples illustrate the above results. We represent a point
$\varphi\in P_\mathbb{R}^*$ by a tuple $(z_i:=\varphi(\omega_i))_{i\in I}$. We
also use the standard numeration of simple roots and fundamental weights from
\cite{bourbaki}.

The construction carried on in this paper, as it will be explained in details in
Section \ref{sec:bipartite}, coincides with the one in \cite{y-system} and \cite{associahedra}
when $c$ is a bipartite Coxeter element.  Therefore the first example in which
something interesting arises is $c=s_1s_2s_3$ in type $A_3$.  In this case $\Pi(c)$
consists of two $\tau_c^\Pi$-orbits:
\begin{center}
\begin{tikzpicture}[description/.style={fill=white,inner sep=2pt}]
	\matrix (m) [matrix of math nodes, row sep=3em, column sep=2.5em, text
		height=1.5ex, text depth=0.25ex]
		{ \omega_1 & -\omega_1+\omega_2 & -\omega_2+\omega_3 & -\omega_3 & \omega_3
		& -\omega_1 \\ };
	\path[->] (m-1-1) edge node[auto] {$ \tau_c $} (m-1-2);
	\path[->]	(m-1-2) edge node[auto] {$ \tau_c $} (m-1-3);
	\path[->]	(m-1-3) edge node[auto] {$ \tau_c $} (m-1-4);
	\path[->]	(m-1-4) edge node[auto] {$ \tau_c $} (m-1-5);
	\path[->]	(m-1-5) edge node[auto] {$ \tau_c $} (m-1-6);
	\path[->] (m-1-6) edge [distance=1cm, bend left ] node[auto] {$ \tau_c $}
		(m-1-1);
\end{tikzpicture}
\end{center}
and 
\begin{center}
\begin{tikzpicture}[description/.style={fill=white,inner sep=2pt}]
	\matrix (m) [matrix of math nodes, row sep=3em, column sep=2.5em, text
		height=1.5ex, text depth=0.25ex]
		{ \omega_2 & -\omega_1+\omega_3 & -\omega_2 \\ };
	\path[->] (m-1-1) edge node[auto] {$ \tau_c $} (m-1-2);
	\path[->]	(m-1-2) edge node[auto] {$ \tau_c $} (m-1-3);
	\path[->] (m-1-3) edge [distance=1cm, bend left ] node[auto] {$ \tau_c $}
	(m-1-1);
\end{tikzpicture}
\end{center}
It is not surprising that the number of orbits and their lengths are the same as
the $A_3$ example in \cite{associahedra}: they depend only on the type of the
cluster algebra and not on the choice of a Coxeter element. Since in this case
$w_0\omega_1=-\omega_3$ we have $1^*=3$ therefore we need to impose 
$f(1)=f(3)$; condition (\ref{thm-pi:polytopality:2}) in
Theorem \ref{thm-pi:polytopality} becomes
$$
	0<f(1)<f(2)<2f(1)
$$
and the corresponding polytope $\asso_c^f(W)$ is defined by the inequalities
$$
	\max\left\{
	z_1,-z_1+z_2,-z_2+z_3,-z_3,z_3,-z_1
	\right\}\le f(1)
$$
$$
	\max\left\{ z_2,-z_1+z_3,-z_2 \right\}\le f(2)
	\,.
$$

This polytope is shown in Figure \ref{fig:A_3}. Note that, to make pictures
easier to plot and view, the angles between fundamental weights are not drawn to
scale, and each facet is labeled by the weight it is orthogonal to.

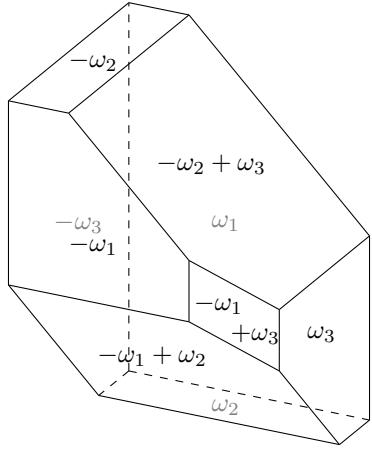
\begin{figure}[h]
	\begin{center}
		\begin{tikzpicture}[line join=round, scale=0.98]
			\draw[dashed](-.816,-1.5)--(-.816,3.5);
			\draw[dashed](-.816,-1.5)--(-1.225,-1.833);
			\draw[dashed](-.816,-1.5)--(2.449,-2.167);
			\draw[-](-.816,3.5)--(-2.449,2.167);
			\draw[-](-.816,3.5)--(0,3.333);
			\draw[-](-1.225,-1.833)--(-2.449,-.333);
			\draw[-](-1.225,-1.833)--(2.041,-2.5);
			\draw[-](-1.633,2)--(-2.449,2.167);
			\draw[-](-1.633,2)--(0,0);
			\draw[-](-1.633,2)--(0,3.333);
			\draw[-](-2.449,-.333)--(-2.449,2.167);
			\draw[-](-2.449,-.333)--(0,-.833);
			\draw[-](0,-.833)--(0,0);
			\draw[-](0,-.833)--(1.225,-1.5);
			\draw[-](0,0)--(1.225,-.667);
			\draw[-](0,3.333)--(2.449,.333);
			\draw[-](1.225,-.667)--(1.225,-1.5);
			\draw[-](1.225,-.667)--(2.449,.333);
			\draw[-](1.225,-1.5)--(2.041,-2.5);
			\draw[-](2.041,-2.5)--(2.449,-2.167);
			\draw[-](2.449,-2.167)--(2.449,.333);
			\draw[gray] (0.5,0.5) node {$\omega_1$};
			\draw[gray] (0.5,-2) node {$\omega_2$};
			\draw[gray] (-1.5,0.5) node {$-\omega_3$};
			\draw (-1.3,2.7) node {$-\omega_2$};
			\draw (1.8,-1) node {$\omega_3$};
			\draw (-1.3,0.2) node {$-\omega_1$};
			\draw (-0.5,-1.3) node {$-\omega_1+\omega_2$};
			\draw (0.3,1.3) node {$-\omega_2+\omega_3$};
			\draw (0.4,-0.6) node {$-\omega_1$};
			\draw (0.9,-1) node {$+\omega_3$};
		\end{tikzpicture}
	\end{center}
	\caption{$\asso_c^f(W)$ in type $A_3$ for $c=s_1s_2s_3$}
	\label{fig:A_3}
\end{figure}

Now let $c=s_1s_2s_3$ in type $C_3$. Then the set $\Pi(c)$ consists of three
orbits:
\begin{center}
\begin{tikzpicture}[description/.style={fill=white,inner sep=2pt}]
	\matrix (m) [matrix of math nodes, row sep=3em, column sep=2.5em, text
		height=1.5ex, text depth=0.25ex]
		{ \omega_1 & -\omega_1+\omega_2 & -\omega_2+\omega_3 & -\omega_1 \\ };
	\path[->] (m-1-1) edge node[auto] {$ \tau_c $} (m-1-2);
	\path[->]	(m-1-2) edge node[auto] {$ \tau_c $} (m-1-3);
	\path[->]	(m-1-3) edge node[auto] {$ \tau_c $} (m-1-4);
	\path[->] (m-1-4) edge [distance=1cm, bend left ] node[auto] {$ \tau_c $}
		(m-1-1);
\end{tikzpicture}
\end{center}
\begin{center}
\begin{tikzpicture}[description/.style={fill=white,inner sep=2pt}]
	\matrix (m) [matrix of math nodes, row sep=3em, column sep=2.5em, text
		height=1.5ex, text depth=0.25ex]
		{ \omega_2 & -\omega_1+\omega_3 & -\omega_1-\omega_2+\omega_3 & -\omega_2 \\ };
	\path[->] (m-1-1) edge node[auto] {$ \tau_c $} (m-1-2);
	\path[->]	(m-1-2) edge node[auto] {$ \tau_c $} (m-1-3);
	\path[->]	(m-1-3) edge node[auto] {$ \tau_c $} (m-1-4);
	\path[->] (m-1-4) edge [distance=1cm, bend left ] node[auto] {$ \tau_c $}
		(m-1-1);
\end{tikzpicture}
\end{center}
\begin{center}
\begin{tikzpicture}[description/.style={fill=white,inner sep=2pt}]
	\matrix (m) [matrix of math nodes, row sep=3em, column sep=2.5em, text
		height=1.5ex, text depth=0.25ex]
		{ \omega_3 & -2\omega_1+\omega_3 & -2\omega_2+\omega_3 & -\omega_3 \\ };
	\path[->] (m-1-1) edge node[auto] {$ \tau_c $} (m-1-2);
	\path[->]	(m-1-2) edge node[auto] {$ \tau_c $} (m-1-3);
	\path[->]	(m-1-3) edge node[auto] {$ \tau_c $} (m-1-4);
	\path[->] (m-1-4) edge [distance=1cm, bend left ] node[auto] {$ \tau_c $}
		(m-1-1);
\end{tikzpicture}
\end{center}
Condition (\ref{thm-pi:polytopality:2}) in
Theorem \ref{thm-pi:polytopality} reads
$$
	f(2)<2f(1)
$$
$$
	f(1)+f(3)<2f(2)
$$
$$
	f(2)<f(3)
$$
as in the corresponding example in \cite{associahedra}. The polytope is given by
the inequalities
$$
	\max\left\{ z_1,-z_1+z_2,-z_2+z_3,-z_1 \right\}<f(1)
$$
$$
	\max\left\{ z_2,-z_1+z_3,-z_1-z_2+z_3,-z_2 \right\}<f(2)
$$
$$
	\max\left\{ z_3,-2z_1+z_3,-2z_2+z_3,-z_3 \right\}<f(3)
$$
and it is shown in Figure \ref{fig:C_3} using the same conventions of Figure
\ref{fig:A_3}.

\begin{figure}[h]
	\begin{center}
		\begin{tikzpicture}[line join=round, scale=0.345]
			\draw[-](-.408,2.167)--(1.225,.167);
			\draw[-](-.408,2.167)--(2.858,4.833);
			\draw[-](-2.858,-1.5)--(-9.39,-.167);
			\draw[-](-2.858,-1.5)--(5.307,-4.833);
			\draw[-](-2.858,1.833)--(-2.858,-1.5);
			\draw[-](-2.858,1.833)--(-4.491,3.833);
			\draw[-](-3.674,9.5)--(2.858,4.833);
			\draw[-](-4.491,3.833)--(-.408,2.167);
			\draw[-](-5.307,9.833)--(-3.674,9.5);
			\draw[dashed](-5.307,9.833)--(-5.307,-3.5);
			\draw[dashed](-6.124,-4.167)--(-5.307,-3.5);
			\draw[-](-6.124,-4.167)--(-9.39,-.167);
			\draw[-](-7.757,6.167)--(-3.674,9.5);
			\draw[-](-7.757,6.167)--(-4.491,3.833);
			\draw[-](-9.39,-.167)--(-9.39,6.5);
			\draw[-](-9.39,6.5)--(-5.307,9.833);
			\draw[-](-9.39,6.5)--(-7.757,6.167);
			\draw[-](1.225,.167)--(-2.858,1.833);
			\draw[-](5.307,-3.167)--(1.225,.167);
			\draw[-](5.307,-4.833)--(5.307,-3.167);
			\draw[-](5.307,-4.833)--(7.757,-6.167);
			\draw[-](7.757,-4.5)--(5.307,-3.167);
			\draw[-](7.757,-4.5)--(9.39,-3.167);
			\draw[-](7.757,-6.167)--(7.757,-4.5);
			\draw[-](8.573,-7.167)--(-6.124,-4.167);
			\draw[-](8.573,-7.167)--(7.757,-6.167);
			\draw[-](9.39,-3.167)--(2.858,4.833);
			\draw[-](9.39,-3.167)--(9.39,-6.5);
			\draw[dashed](9.39,-6.5)--(-5.307,-3.5);
			\draw[-](9.39,-6.5)--(8.573,-7.167);
			\draw[gray] (-7,3) node {$-\omega_3$};
			\draw[gray] (1,-5.3) node {$\omega_2$};
			\draw[gray] (0,0) node {$\omega_1$};
			\draw (8.6,-5.3) node {$\omega_3$};
			\draw (-7,7.68) node {$-\omega_2$};
			\draw (-6.2,2) node {$-\omega_1$};
			\draw (-3,-3.3) node {$-\omega_1+\omega_2$};
			\draw (-2.8,6) node {$-2\omega_2+\omega_3$};
			\draw (4,0) node {$-\omega_2+\omega_3$};
			\draw (-.2,-1.3) node {$-2\omega_1+\omega_3$};
			\draw (6.2,-4.2) node {\tiny$-\omega_1$};
			\draw (6.5,-4.8) node {\tiny$+\omega_3$};
			\draw (-3,2.8) node {\tiny$-\omega_1$};
			\draw (-2,2.2) node {\tiny$-\omega_2$};
			\draw (-1,1.6) node {\tiny$+\omega_3$};
		\end{tikzpicture}
	\end{center}
	\caption{$\asso_c^f(W)$ in type $C_3$ for $c=s_1s_2s_3$}
	\label{fig:C_3}
\end{figure}
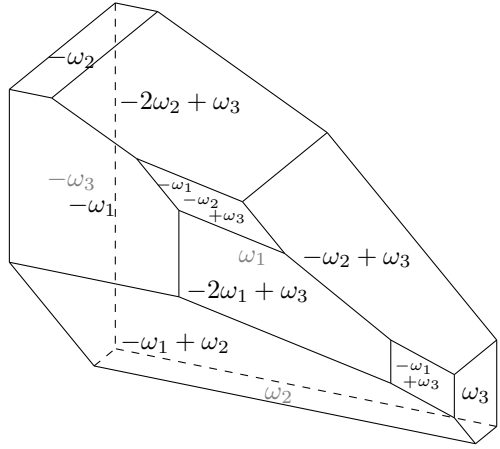

To prove the results we discussed so far will use two types of argument. The first
one is induction on the rank of $I$. Unfortunately the set $\Pi(c)$, and in
general the whole weight lattice $P$, does not behave nicely when considering
sub-diagrams of $I$. It is then convenient to introduce an auxiliary set of
labels: the \emph{$c$-almost positive roots}:
\begin{equation*}
	\Phi_{\ap}(c):=(c^{-1}-1)\Pi(c)
\end{equation*}
whose behaviour is more manageable. On the one hand the new set is related to
the old one by a linear transformation therefore any property proved for
$\Phi_{\ap}(c)$ can be transported back to $\Pi(c)$. 

On the other hand $\Phi_{\ap}(c)$ is modeled after the set $\Phi_{\ge-1}$
introduced in \cite{y-system}. It differs from the latter in several respects:
first it still consists of $g$-vectors (in an odd-looking basis) and not
denominator vectors; second it contains all the positive roots (as
$\Phi_{\ge-1}$ does) but the negative simples are replaced by other negative
roots depending on the choice of the Coxeter element $c$. However, contrary to
what happens for $\Pi(c)$, it retains a notion of subset corresponding to a
Dynkin sub-diagram. 
In order to use induction on $|I|$ it will then suffice to show that the
$c$-compatibility degree on $\Phi_{\ap}(c)$ is preserved when restricting
to a sub-diagram of $I$ (this is the content of Proposition
\ref{prop:inductive-property}).

To explain the second type of argument we need an observation on Coxeter
elements. For a given Coxeter element $c$, we call a simple reflection $s_i$
\emph{initial} (resp. \emph{final}) if $c$ admits a reduced expression of the
form $c=s_iv$ (resp. $c=vs_i$). Conjugating any Coxeter element by an initial or
final reflection produces another Coxeter element; call such a conjugation an
\emph{elementary move} and call two Coxeter elements related by a single
elementary move \emph{adjacent}.
The following is a well known fact. 
\begin{lemma}
	Any Coxeter element can be reached from any other via a sequence of elementary
	moves.
	\label{lemma:coxeter-connected}
\end{lemma}
A proof can be found in \cite{pfeiffer} Theorem 3.1.4.

We will construct maps $\sigma_i^{\pm1}$ relating sets of $c$-almost positive
roots for adjacent Coxeter elements.  These maps will not be linear so, a
priori, they might not preserve all the properties we are interested into.  Our
strategy will be to show that, for any Coxeter element $c$, there exists a
bipartite Coxeter element $t$ and a sequence of elementary moves relating the
two, such that all the corresponding maps $\sigma_i^{\pm1}$ preserve the desired
properties.  This will reduce our statements to the bipartite case. Our results
will then follow from another important property of the set of $c$-almost
positive roots: when the Coxeter element is bipartite, there exists bijection 
$$
	t_-:\Phi_{\ge-1}\rightarrow\Phi_{\ap}(t)
$$
which is induced by a linear map. This will allow us, in this particular case,
to deduce our results from their analogs from \cite{y-system} and \cite{shih}.

As a byproduct of the construction we get an explicit description of the
exchange relations of $\mathcal{A}_0(c)$. 
Two cluster variables $x_{\lambda,c}$ and $x_{\mu,c}$ in it are exchangeable if and only if 
$(\lambda||\mu)^\Pi_c=(\mu||\lambda)_c^\Pi=1$.
Denote by $T$ the cyclic group
generated by $\tau_c^\Pi$.  The proof of Theorem \ref{thm-pi:polytopality}
relies on the fact that, except in some degenerate cases, for any pair of
weights $\lambda$ and $\mu$ in $\Pi(c)$ corresponding to a pair of exchangeable
cluster variables, the set
\begin{equation*}
	\left\{ \tau\left( \tau^{-1}(\lambda)+\tau^{-1}(\mu) \right)
	\right\}_{\tau\in T}
\end{equation*}
consists of two vectors: $\lambda+\mu$ and another one denoted by
$\lambda\uplus_c\mu$. 

Use Theorem \ref{thm-pi:complete-fan} to label all cluster monomials in
$\mathcal{A}_0(c)$ by points of $P$:  
\begin{equation*}
	x_{\sum m_\lambda\lambda,c}:=\prod x_{\lambda,c}^{m_\lambda}.
\end{equation*}

\begin{theorem}
	All the exchange relations in $\mathcal{A}_0(c)$ are of the form
	\begin{equation*}
		x_{\lambda,c} x_{\mu,c}=x_{\lambda+\mu,c}+x_{\lambda\uplus_c\mu,c}
	\end{equation*}
	\label{thm-pi:exchange-relations}
\end{theorem}

We now discuss the connection of $\mathcal{F}_c^\Pi$ with the Cambrian fan
defined in \cite{cambrian-lattices}. First recall some definitions and results
from \cite{sortable-polytopality}.

Let $D$ be the fundamental Weyl chamber, i.e., the $\mathbb{R}_+$-span of the
fundamental weights. The \emph{Coxeter fan} $\mathcal{F}$ is the complete
simplicial fan in $P_\mathbb{R}$ whose maximal cones are the images of $D$ under
the action of $W$. It is well known that the correspondence 
$$
	w\mapsto w(D)
$$
is a bijection between $W$ and the set of maximal cones of $\mathcal{F}$;
moreover $\mathcal{F}$ is the normal fan to a distinguished polytope: the
permutahedron (see e.g. \cite{postnikov}).

Using the (right) weak order,  $W$ can be regarded  as a lattice with minimal and
maximal element $e$ and $w_0$ respectively. To each lattice congruence on $W$
corresponds a fan that  coarsens $\mathcal{F}$ as shown in
\cite{fan-from-congruence}; maximal cones in the new fan are obtained glueing
together cones of $\mathcal{F}$ corresponding to elements of $W$ belonging to
the same equivalence class.

Fix a Coxeter element $c$ and one of its reduced expressions. For any subset
$J\subset I$, denote by $c_J$ the sub-word of $c$ obtained omitting the simple
reflections $\left\{ s_i \right\}_{ i\in I\setminus J}$. Let $c^\infty$ be the
formal word obtained concatenating infinitely many copies of $c$. 
Every reduced expression of $w\in W$ can be seen as a sub-word of $c^\infty$; call the
\emph{$c$-sorting word} of $w$ the lexicographically first sub-word of
$c^\infty$  realizing it.  The $c$-sorting word of $w$ can be encoded by a
sequence of subsets $I_1,I_2,\dots I_k$ of $I$ (the \emph{$c$-factorization} of
$w$) so that
\begin{equation*}
	w=c_{I_1}c_{I_2}\cdots c_{I_k}
	\,.
\end{equation*}
Note that the $c$-factorization of $w$ is independent on the reduced expression
chosen for $c$: it depends only on the Coxeter element itself.

\begin{definition}
	An element $w$ in $W$ is
	\begin{itemize}
		\item 
			\emph{$c$-sortable} if its $c$-factorization is such that
			$$
				I_1\supseteq I_2 \supseteq\cdots \supseteq I_k
			$$

		\item
			\emph{$c$-antisortable} if $w w_0$ is $c^{-1}$-sortable

	\end{itemize}
\end{definition}

As an example pick $c=s_1s_2s_3$ in type $A_3$, then $s_2s_3s_2$ is $c$-sortable
with the $c$-factorization $\left\{ 2,3 \right\},\left\{ 2\right\}$, the element
$s_2s_3s_1s_2s_1$ is $c$-antisortable while $s_2s_3s_2s_1$ is neither.

For any element $w$ in $W$, again in the weak order, 
there exist a unique minimal $c$-antisortable element above it and a unique
maximal $c$-sortable below it; denote them by $\pi_c^\uparrow(w)$ and
$\pi^c_\downarrow(w)$ respectively.

\begin{proposition}[cf. \cite{sortable-cambrian}]
	For any $w\in W$ the sets
	$$
		\left(\pi^c_\downarrow\right)^{-1}\left(\pi^c_\downarrow(w)\right)
	$$ 
	and
	$$
		\left(\pi_c^\uparrow\right)^{-1}\left(\pi_c^\uparrow(w)\right)
	$$ 
	coincide; they are intervals in the lattice $W$ with minimal element
	$\pi^c_\downarrow(w)$ and maximal element $\pi_c^\uparrow(w)$.
\end{proposition}

Define a lattice congruence on $W$ by setting
\begin{equation}
	v\sim w 
	\quad \Leftrightarrow \quad
	\pi^c_\downarrow(v)=\pi^c_\downarrow(w)
	\,.
	\label{eqn:sortable-equivalence}
\end{equation}
The $c$-Cambrian fan $\mathcal{F}_c^C$ (defined in \cite{cambrian-fan}) is the
complete simplicial fan obtained from $\mathcal{F}$ by coarsening with respect
to the lattice congruence (\ref{eqn:sortable-equivalence}); its maximal cones
are parametrized by $c$-sortable elements.

In \cite{sortable-polytopality} it was shown that, for any point $a$ in the
fundamental Weyl chamber, there is a unique simple polytope $\asso_c^a(W)$ with
normal fan $\mathcal{F}_c^C$ and such that $a$ is a vertex of $\asso_c^a(W)$.

We have now all the required notations to state our last result.
\begin{theorem}
	For every $f:I\rightarrow\mathbb{R}$ satisfying the hypothesis of Theorem
	\ref{thm-pi:polytopality} there exists a point $a\in D$ such that the
	polytopes $\asso_c^a(W)$ and $\asso_c^f(W)$ coincide.
	\label{thm:polytopes-are-equal}
\end{theorem}

As a direct consequence we get
\begin{corollary}
	The $c$-Cambrian fan $\mathcal{F}_c^C$ and the $c$-cluster fan
	$\mathcal{F}_c^\Pi$ coincide.
\end{corollary}

\section{The set $\Phi_{\ap}(c)$}
\label{sec:phi-ap}
Fix a Dynkin diagram $I$ and let $\Phi=\Phi_+\sqcup\Phi_-$ be the corresponding
root system.  For convenience we identify $I$ with $\left\{ 1,\dots n \right\}$
so that a chosen Coxeter element is $c=s_1\cdots s_n$.

For $i\in I$ let 
\begin{equation}
	\beta_i^c:=s_n\cdots s_{i+1}\alpha_i 
	\,.
	\label{eqn:root_basis}
\end{equation}
\begin{remark}
	It is known that the roots (\ref{eqn:root_basis}) are exactly the positive roots that are
	mapped into negative  roots by $c$; moreover they form a $\mathbb{Z}$-basis of
	the root lattice $Q$ since the linear map sending each $\alpha_i$ to
	$\beta_i^c$ is unitriangular.
	\label{rk:triangular-change-of-base}
\end{remark}

We call $\Phi_{\ap}(c):=\Phi_+\cup \left\{-\beta_i^c\right\}_{i\in I}$ the set
of the \emph{$c$-almost-positive roots} and define a bijection
$\tau_c^\Phi:\Phi_{\ap}(c)\rightarrow\Phi_{\ap}(c)$ by setting, for
$\alpha\in\Phi_{\ap}(c)$,
\begin{equation*}
	\tau^\Phi_c(\alpha):=\left\{ 
	\begin{array}{ll}
		-\beta_i^c & \mbox{if } \alpha=\beta_i^c\\
		c\alpha & \mbox{otherwise}\,.
	\end{array}
	\right.
\end{equation*}

\begin{definition}
	The \emph{$c$-compatibility degree} on $\Phi_{\ap}(c)$ is the unique
	$\tau^\Phi_c$-invariant function
	\begin{equation*}
		(\bullet||\bullet)^\Phi_c:\Phi_{\ap}(c)\times\Phi_{\ap}(c)\longrightarrow\mathbb{N}
	\end{equation*}
  defined by the initial conditions
	\begin{equation*}
		(-\beta_i^c||\alpha)^\Phi_c:=\left[ \alpha;\alpha_i \right]_+
		\,.
	\end{equation*}
\end{definition}

These definitions are justified by the following proposition.
\begin{proposition}
  The	linear map 
	\begin{equation*}
		\phi_c:=(c^{-1}-1):P_\mathbb{R}\longrightarrow Q_\mathbb{R}
	\end{equation*}
	is invertible and restricts to an isomorphism of the weight lattice $P$ with
	the root lattice $Q$ sending $\Pi(c)$ to $\Phi_{\ap}(c)$.
	Moreover $\phi_c$ intertwines $\tau_c^\Pi$ and $\tau_c^\Phi$ and
	transform the compatibility degree $(\bullet||\bullet)_c^\Pi$ on $\Pi(c)$
	into the compatibility degree $(\bullet||\bullet)_c^\Phi$ on $\Phi_{\ap}(c)$.
  \label{prop:translation-pi-phi}
\end{proposition}
\begin{proof}
	To show that $\phi_c$ is a lattice isomorphism, in view of Remark
	\ref{rk:triangular-change-of-base}, it suffices to establish that  
	\begin{equation*}
		\phi_c(\omega_i)=-\beta_i^c
		\,.
	\end{equation*}
	Using the well-known property
	\begin{equation*}
		s_i\omega_j=\left\{
		\begin{array}{ll}
			\omega_i-\alpha_i & \mbox{if } i=j\\
			\omega_j & \mbox{otherwise,}
		\end{array}
		\right.
	\end{equation*}
	we have
	$$
  	\phi_c(\omega_i)=s_n\cdots s_1\omega_i-\omega_i=s_n\cdots
  	s_{i+1}(s_i\omega_i-\omega_i)=s_n\cdots s_{i+1}(-\alpha_i)=-\beta_i^c
		.
	$$

	The sets $\Pi(c)$ and $\Phi_{\ap}(c)$ have the same cardinality.  Indeed
	Proposition 1.7 in \cite{shih} states that, for every $i$, the sum
	$h(i,c)+h(i^*,c)$ is equal to the Coxeter number $h$, hence
	$$
		|\Pi(c)|=
		\sum_{i\in I}(h(i,c)+1)=
		\frac{1}{2}\sum_{i\in I}(h(i,c)+h(i^*,c)+2)=
		\frac{1}{2}\sum_{i\in I}(h+2)=
		|\Phi_{\ap}(c)|
		\,.
	$$

	To conclude the proof of the first part it suffices to check that any weight
	in $\Pi(c)\setminus\left\{ \omega_i \right\}_{i\in I}$ is mapped to a positive
	root. This was already showed in \cite{shih} during the proof of the
	inequalities (1.8) in it.

	To show that, for any $\alpha\in\Phi_{\ap}(c)$, 
	$$
		\phi_c^{-1}\left(\tau^\Phi_c(\alpha)\right)=
		\tau_c^\Pi\left(\phi_c^{-1}(\alpha)\right)
	$$
	there are two cases to consider:
	\begin{enumerate}
		\item
			if $\alpha=\beta_i^c$ then 
			\begin{align*}
  			\phi_c^{-1}\left(\tau^\Phi_c(\beta_i^c)\right)=
  			\phi_c^{-1}\left(-\beta_i^c\right)=
				\omega_i=
  			\tau_c^\Pi(-\omega_{i})=
  			\tau_c^\Pi\left(\phi_c^{-1}\left(\beta_i^c\right) \right)
			\end{align*}

		\item
			if $\alpha\neq\beta_i^c$ for any $i$ then
			$$
  			\phi_c^{-1}\left(\tau^\Phi_c(\alpha)\right)=
  			\phi_c^{-1}\left(c\alpha\right)=
  			(c^{-1}-1)^{-1}c\alpha=
  			c(c^{-1}-1)^{-1}\alpha=
  			\tau_c^\Pi\left(\phi_c^{-1}\left(\alpha\right)\right)
				\,.
  		$$
	\end{enumerate}
	To conclude the proof it is sufficient to show that both compatibility degrees
	satisfy the same initial conditions. On the one hand we have 
  $$
  	(-\beta_i^c||\alpha)^\Phi_c=[\alpha;\alpha_i]_+
  $$
  and on the other
  \begin{align*}
    \left(\phi_c^{-1} \left(-\beta_i^c\right)||\phi_c^{-1}\left( \alpha\right)\right)_c^\Pi=
  	\left( \omega_i||\phi_c^{-1}(\alpha)\right)_c^\Pi=
  	\left[(c^{-1}-1)(c^{-1}-1)^{-1} \alpha;\alpha_i \right]_+
		\,.
  \end{align*}

\end{proof}

\begin{remark}
	As in the case of $\Pi(c)$ the action of $\tau_c^\Phi$ on $\Phi_{\ap}(c)$ and
	the action of $w_0$ on $I$ are compatible, i.e. there exist $m\in\mathbb{Z}$
	such that  
	$$
		\left(  \tau_c^\Phi\right)^m(-\beta_i^c)=-\beta_j^c 
	$$
	if and only if	$j=i$ or $j=i^*$.
	\label{rk:wheels}
\end{remark}

We can now rephrase Theorems \ref{thm-pi:z-basis}, \ref{thm-pi:complete-fan}, 
\ref{thm-pi:polytopality}, \ref{thm-pi:exchange-relations}, and Proposition
\ref{prop-pi:unique-cluster-expansion} in this new setup. 

Let $\Delta_c^\Phi$ be the abstract simplicial complex having elements of
$\Phi_{\ap}(c)$ as vertices and with subsets of pairwise compatible roots as
simplices; similarly to the case of $\Pi(c)$, we call \emph{$c$-clusters} the
maximal (by inclusion) simplices.

In view of Proposition \ref{prop:translation-pi-phi}, Theorem
\ref{thm-pi:z-basis} is equivalent to the following.

\begin{theorem}
	Each $c$-cluster in $\Delta_c^\Phi$ is a $\mathbb{Z}$-basis of the root
	lattice $Q$. 
	\label{thm-phi:z-basis}
\end{theorem}

\begin{definition}
	For any $\gamma$ in $Q$ we call a \emph{$c$-cluster expansion} of $\gamma$ an
	expression
	\begin{equation*}
		\gamma=\sum_{\alpha\in\Phi_{\ap}(c)}m_\alpha\alpha
	\end{equation*}
	where all the coefficients $m_\alpha$ are nonnegative integers such that $m_\alpha
	m_\delta=0$ whenever $\left( \alpha||\delta \right)^\Phi_c\neq0$.
\end{definition}

The counterpart of Proposition \ref{prop-pi:unique-cluster-expansion} is the
following:
\begin{proposition}
	Any $\gamma$ in the root lattice $Q$ admits a unique $c$-cluster expansion.
	\label{prop-phi:unique-cluster-expansion}
\end{proposition}

\begin{remark}
	Our proof of Proposition \ref{prop-phi:unique-cluster-expansion} will mimic, step by
	step, the proof of Theorem 3.11 in \cite{associahedra}.  A sketch of a
	different proof, more similar to the others in this paper, will be also given.
\end{remark}

Let $\mathcal{F}_c^\Phi$ be the set of all the cones in the space $Q_\mathbb{R}$
that are the positive linear span of simplices of the complex $\Delta_c^\Phi$.
A direct consequence of Proposition \ref{prop-phi:unique-cluster-expansion} is
the following counterpart of Theorem \ref{thm-pi:complete-fan}.

\begin{theorem}
	$\mathcal{F}_c^\Phi$ is a complete simplicial fan. 	
	\label{thm-phi:complete-fan}
\end{theorem}

As for the case of $\Pi(c)$, once Theorem \ref{thm-phi:complete-fan} is
established, any function defined on $\Phi_{\ap}(c)$ can be extended to a
continuous, piecewise linear function on $Q_\mathbb{R}$ that is linear on the 
maximal cones of $\mathcal{F}_c^\Phi$. In particular, any function 
$$
	f:I\longrightarrow\mathbb{R}
$$
such that $f(i)=f(i^*)$ gives rise to a $\tau_c^\Phi$-invariant, continuous,
piecewise-linear function 
$$
	F_c=F_{c;f}:Q_\mathbb{R}\longrightarrow\mathbb{R}
$$
by setting
\begin{equation*}
	F_c(-\beta_i^c):=f(i)
\end{equation*}
and extending, first to $\Phi_{\ap}(c)$ and then to $Q_\mathbb{R}$, as
prescribed.

Let $\asso_c^{f,\Phi}(W)$ be the subset of $Q_\mathbb{R}^*$ defined by
\begin{equation}
	\asso_c^{f,\Phi}(W):=
	\left\{ \varphi\in Q_\mathbb{R}^* \mid \varphi(\alpha)\le F_c(\alpha)\quad\forall
	\alpha\in\Phi_{\ap}(c)\right\}
	\,.
\end{equation}
\begin{theorem}
	If $f:I\rightarrow\mathbb{R}$ is such that
	\begin{enumerate}
		\item 
			for any $i\in I$
			$$
				f(i)=f(i^*)
			$$
		\item
			for any $j\in J$
			$$
				\sum_{i\in I}a_{ij}f(i)>0
			$$
	\end{enumerate}
	then $\asso_c^{f,\Phi}(W)$ is a simple $n$-dimensional polytope with support
	function $F_c$.  Furthermore, the domains of linearity of $F_c$ are exactly
	the maximal cones of $\mathcal{F}_c^\Phi$, hence the normal fan of
	$\asso_c^{f,\Phi}(W)$ is $\mathcal{F}_c^\Phi$.
	\label{thm-phi:polytopality}
\end{theorem}
Again by Proposition \ref{prop:translation-pi-phi}, Theorem
\ref{thm-phi:polytopality} implies Theorem \ref{thm-pi:polytopality}.

The proof of Theorem \ref{thm-phi:polytopality} is based on an explicit
characterization of the roots in $\Phi_{\ap}(c)$ belonging to adjacent maximal
cones of $\mathcal{F}_c^\Phi$. Namely there exist two $c$-cluster $C_\alpha$ and
$C_\gamma$ such that $C_\alpha\setminus\left\{ \alpha
\right\}=C_\gamma\setminus\left\{ \gamma \right\}$ if and only if
$$
	(\alpha||\gamma)_c^\Phi=1=(\gamma||\alpha)_c^\Phi
$$
(Cf. Lemma \ref{lemma:reduction-to-c.e.}). For all such pairs of roots the set 
$$
	\left\{ 
		\left( \tau_c^\Phi \right)^{-m}
		\left( 
			\left( \tau_c^\Phi \right)^m
			(\alpha)
			+
			\left( \tau_c^\Phi \right)^m
			(\beta)
		\right)
	\right\}_{m\in\mathbb{Z}}
$$
consists (when $I$ has no connected component with only one node) of precisely
two vectors, $\alpha+\gamma$ and $\alpha\uplus_c\gamma$; their $c$-cluster
expansion are supported on $C_\alpha\cup C_\gamma$ and they are disjoint (Cf.
Proposition \ref{prop:2-element-orbit} and Corollary \ref{cor:c.e.-intersection}).

Let $\mathcal{A}_0(c)$ the coefficient-free cluster algebra with initial
orientation given by $c$; label its cluster variables by roots in
$\Phi_{\ap}(c)$ and, in view of Proposition
\ref{prop-phi:unique-cluster-expansion}, its cluster monomials by points in the
root lattice. Using this notation Theorem \ref{thm-pi:exchange-relations} can be
restated as follows.
\begin{theorem}
	All the exchange relations in $\mathcal{A}_0(c)$ are of the form 
	$$
		x_{\alpha,c} x_{\gamma,c}=x_{\alpha+\gamma,c}+x_{\alpha\uplus_c\gamma,c}
	$$
	for suitable $c$-almost positive roots such that 
	$$
	  (\alpha||\gamma)_c^\Phi=1=(\gamma||\alpha)_c^\Phi.
	$$
	\label{thm-phi:exchange-relations}
\end{theorem}

As mentioned before the main advantage of the labels $\Phi_{\ap}(c)$ over
$\Pi(c)$ is that it is easyer to set up inductions on $|I|$.
Let $J\subset I$ be a sub diagram of $I$. Fix a Coxeter element $c$ for $I$ and
denote by $c_J$ the sub-word of $c$ obtained omitting all the simple reflections
$\left\{s_i \right\}_{i\in I\setminus J}$.  By construction $c_J$ is a Coxeter
element in the Weyl group $W_J$ (we denote by $W_J$ the standard parabolic
subgroup of $W$ generated by $\left\{ s_j \right\}_{j\in J}$). Let
\begin{equation*}
	\iota:\Phi^J_{\ap}(c_J)\longrightarrow\Phi_{\ap}(c)
\end{equation*}
be the ``twisted'' inclusion  map given by
\begin{equation}
	\iota(\alpha):=\left\{ 
	\begin{array}{ll}
		-\beta_i^{c} & \mbox{if } \alpha=-\beta_i^{c_J},\,i\in J\\
		\alpha & \mbox{otherwise.}
	\end{array}
	\right.
\end{equation}

From this moment on, unless it is not clear from the context, superscripts
$\Phi$ and $\Pi$ will be omitted in order to make notation less heavy.

Denote by $\left( \bullet||\bullet \right)^J_{c_J}$ the $c_J$-compatibility
degree on $\Phi^J_{\ap}(c_J)$. The key property is this:
\begin{proposition}
  Let $\alpha$ and $\gamma$ be roots in $\Phi^J_{\ap}(c_J)$. Then
	\begin{equation*}
		\left( \iota(\alpha)||\iota(\gamma) \right)_c=
		\left( \alpha||\gamma \right)^J_{c_J}
		\,.
	\end{equation*}
  \label{prop:inductive-property}
\end{proposition}

\begin{remark}
	In the setup of almost positive roots the analog of this statement is point 3
	of Proposition 3.3 in \cite{y-system}; there the map $\iota$ is the ordinary
	inclusion.  
	A proof of Proposition \ref{prop:inductive-property} will be given
	in Section \ref{sec:technical-results}. 
\end{remark}

The original construction in \cite{y-system} does not distinguish among the
possible bipartite orientations of $I$. With this motivation in mind 
consider the map $\alpha\mapsto\overline\alpha$ between $\Phi_{\ap}(c)$ and
$\Phi_{\ap}(c^{-1})$ defined by
\begin{equation}
	\overline\alpha:=\left\{ 
	\begin{array}{ll}
		-\beta_i^{c^{-1}} & \mbox{if } \alpha=-\beta_i^{c},\,i\in I\\
		\alpha & \mbox{otherwise}\,.
	\end{array}
	\right.
\end{equation}

\begin{proposition}
  For any $\alpha$ and $\gamma$ in $\Phi_{\ap}(c)$
	\begin{equation*}
		\left( \alpha||\gamma \right)_c=
		\left( \overline\alpha||\overline\gamma \right)_{c^{-1}}
		\,.
	\end{equation*}
  \label{prop:inverse-invariance}
\end{proposition}
\begin{proof}
  Initial conditions agree:
  $$
	  \left( -\beta_i^c||\alpha \right)_c=
 		\left[ \alpha;\alpha_i \right]_+=
 		\left[ \overline\alpha;\alpha_i \right]_+=
 		\left( -\beta_i^{c^{-1}}||\overline\alpha \right)_{c^{-1}}=
  	\left(\overline{-\beta_i^c}||\overline\alpha \right)_{c^{-1}}
		\,.
  $$
  It suffices then to show that, for any $\alpha\in\Phi_{\ap}(c)$
  $$
 		\overline{\tau_c(\alpha)}=\tau_{c^{-1}}^{-1}(\overline\alpha)
		\,.
  $$
  There are three cases to be considered.
	\begin{enumerate}
		\item  
			If $\alpha=-\beta_i^c$ for some $i\in I$ then on the one hand 
			\begin{align*}
				\overline{\tau_c(-\beta_i^c)}= 
				\overline{-c\beta_i^c}=
				\overline{-s_1\cdots s_n(s_n\cdots s_{i+1}\alpha_i)}= 
				\overline{s_1\cdots s_{i-1}\alpha_i}= 
				\beta_i^{c^{-1}}
				;
			\end{align*}
			on the other hand
			\begin{align*}
				\tau_{c^{-1}}^{-1}\left(\overline{-\beta_i^c}\right)=
				\tau_{c^{-1}}^{-1}\left(-\beta_i^{c^{-1}} \right)=
				\beta_i^{c^{-1}}
				\,.
			\end{align*}
 
 		\item
  		When $\alpha=\beta_i^c$
  		\begin{align*}
    		\overline{\tau_c(\beta_i^c)}=
    		\overline{-\beta_i^c}=
    		-\beta_i^{c^{-1}}=
   			s_1\cdots s_i\alpha_i
  		\end{align*}
			multiplying and dividing by $s_{i+1}\cdots s_n$ we get
  		\begin{align*}
		    s_1\cdots s_i(s_{i+1}\cdots s_n s_n\cdots s_{i+1})\alpha_i=
   		 	\left( c^{-1} \right)^{-1}\beta_i^c=
		    \tau_{c^{-1}}^{-1}\left( \beta_i^c \right)=
   		  \tau_{c^{-1}}^{-1}\left(\overline{\beta_i^c} \right)
				\,.
  		\end{align*}

		\item
  		Finally for $\alpha\neq\pm\beta_i^c$
  		$$
  			\overline{\tau_c(\alpha)}=
			  \overline{c\alpha}=
			  c\alpha=
			  (c^{-1})^{-1}\alpha=
			  \tau_{c^{-1}}^{-1}(\overline\alpha)
				\,.
		  $$
	\end{enumerate}

\end{proof}

\section{The bipartite case}
\label{sec:bipartite}
In this section we assume that the Dynkin diagram $I$ is connected; the
statements in the general case are easily reduced to this.  Since any connected
Dynkin diagram $I$ is a tree, we can split $I$ into two disjoint subsets $I_+$
and $I_-$ such that every edge in $I$ has one  endpoint in $I_+$ and one in
$I_-$.  Up to relabeling, this can be done in a unique way. A bipartite Coxeter
can thus be written as
\begin{equation}
	t=t_\varepsilon t_{-\varepsilon}
	\label{eqn:bipartite-coxeter}
\end{equation}
where $\varepsilon$ denotes a sign and
\begin{equation*}
	t_\varepsilon:=\prod_{i\in I_\varepsilon}s_i
\end{equation*}
(the expression makes sense since the factors commute with each other). 
By our assumption there are precisely two bipartite Coxeter elements in $W$: 
$t=t_+t_-$ and $t^{-1}=t_-t_+$.

Let $\Phi_{\ge-1}$ be the set of \emph{almost positive roots}, i.e.
\begin{equation*}
	\Phi_{\ge-1}:=\Phi_+\cup\left\{ -\alpha_i \right\}_{i\in I}
\end{equation*}
introduced in \cite{y-system} to parametrize cluster variables in the special
case of bipartite initial cluster.  On it there are two involutions $\tau_+$ and
$\tau_-$ defined by
\begin{equation*}
	\tau_\varepsilon(\alpha)=
	\left\{ 
	\begin{array}[]{ll}
		\alpha & \mbox{if } \alpha=-\alpha_i \mbox{ and } i\in I_{-\varepsilon}\\
		t_\varepsilon\alpha & \mbox{otherwise}
	\end{array}
	\right.
\end{equation*}
and a unique $\left\{ \tau_+,\tau_- \right\}$-invariant compatibility degree
function $\left( \bullet ||\bullet \right)_{\ge-1}$
satisfying
\begin{equation*}
	\left( -\alpha_i||\gamma \right)_{\ge-1}=
	\left[ \gamma;\alpha_i\right]_+.
\end{equation*}

Call two almost positive roots $\alpha$ and $\gamma$ \emph{compatible} if 
\begin{equation}
	\left( \alpha||\gamma \right)_{\ge-1}=0=\left( \gamma||\alpha
	\right)_{\ge-1}
	\,.
\end{equation}
The \emph{cluster complex} $\Delta_{\ge-1}$ is the abstract simplicial complex
induced on $\Phi_{\ge-1}$ by the compatibility degree function; its simplices
are subsets of pairwise compatible almost positive roots.
As before  call the maximal simplices \emph{clusters} and consider the
set $\mathcal{F}_{\ge-1}$ of all simplicial cones generated by simplices.

As we mentioned in the introduction our construction is based on the results for
the bipartite case given in \cite{y-system} and \cite{associahedra}. From the
first paper we will need the following.
\begin{proposition}$\,$
\begin{enumerate}
	\item \emph{[Proposition 3.3 (2)]}
		For any pair of almost positive roots $\alpha$ and $\gamma$, we have 
		$$
			\left(\alpha||\gamma \right)_{\ge-1}=0
		$$ 
		if and only if $\left( \gamma||\alpha \right)_{\ge-1}=0$.

	\item \emph{[Proposition 3.3 (3)]}
		Let $J$ be a subset of $I$ and denote by $\Phi^J_{\ge-1}$ the corresponding
		set of almost positive roots. Let $\alpha$ and $\gamma$ be roots in
		$\Phi^J_{\ge-1}$, then
		$$
    	\left( \alpha||\gamma \right)_{\ge-1}=
    	\left( \alpha||\gamma\right)^J_{\ge-1}
    $$
		where $(\bullet||\bullet)_{\ge-1}^J$ denotes the compatibility degree
		function on $\Phi^J_{\ge-1}$.

	\item \emph{[Theorem 1.8]}
		Each cluster in the cluster complex is a
		$\mathbb{Z}$-basis of the root lattice $Q$. 
	
	\item \emph{[Theorem 3.11]}
		Any $\gamma\in Q$ admits a unique \emph{cluster expansion}. In other words
		$\gamma$ can uniquely be written as
		$$
			\gamma=\sum_{\alpha\in\Phi_{\ge-1}}m_\alpha\alpha
		$$
		so that all the coefficients $m_\alpha$ are negative integers and $m_\alpha
		m_{\alpha'}=0$ if $\left( \alpha||\alpha' \right)_{\ge-1}\neq0$.

  \item \emph{[Theorem 1.10]} 
		$\mathcal{F}_{\ge-1}$ is a complete simplicial fan in $Q_\mathbb{R}$. 
\end{enumerate}
\end{proposition}

The results we will need form \cite{associahedra} can be summarized as follows.
\begin{proposition}
	Suppose that $I$ has at least $2$ vertices.  Let $\alpha$ and $\gamma$ be
	almost positive roots such that $\left( \alpha||\gamma \right)_{\ge
	-1}=1=\left( \gamma||\alpha \right)_{\ge-1}$. Then we have:
	\begin{enumerate}
		\item \emph{[Theorem 1.14]}
			The set  
			$$
				\left\{ \tau\left( \tau^{-1}(\alpha)+\tau^{-1}(\gamma) \right)
				\right\}_{\tau\in T}
			$$
			where $T$ denotes the group generated by
			$\tau_+$ and $\tau_-$, consists of exactly two elements $\alpha+\gamma$
			and $\alpha\uplus\gamma$.

		\item \emph{[Lemma 2.3]}
			Any root appearing with a positive coefficient in the cluster expansion
			of $\alpha+\gamma$ or $\alpha\uplus\gamma$ is compatible with both
			$\alpha$, $\gamma$, and with any other root compatible with both $\alpha$
			and $\gamma$.
			\label{prop:bipartite-inductive-property_2:compatibility}

		\item \emph{[Lemma 2.4]}
			Let $f:I\rightarrow\mathbb{R}_+$ be any function such that, for any
			$i\in I$,
			$$
				f(i)=f(i^*)
			$$
			and
			$$
				\sum_{i\in I}a_{ij}f(i)>0
			$$
			for any $j\in I$. Let $F_{\ge-1}:Q_\mathbb{R}\rightarrow\mathbb{R}$ be a
			continuous piecewise-linear function on $Q_\mathbb{R}$ that is linear on
			the maximal cones of $\mathcal{F}_{\ge-1}$, invariant under the action
			of $T$, and such that 
			$$
				F_{\ge-1}(-\alpha_i)=f(i)
				\,.
			$$
			Then,
			$$
				F_{\ge-1}(\alpha)+F_{\ge-1}(\gamma)>\max\left\{
				F_{\ge-1}(\alpha+\gamma),F_{\ge-1}(\alpha\uplus\gamma) \right\}
				\,.
			$$
			\label{prop:bipartite-inductive-property_2:inequality}

	\end{enumerate}
\end{proposition}

The statements regarding $\alpha\uplus\gamma$ are not expressed explicitly in
\cite{associahedra} but can be recovered immediately from the corresponding
statements about $\alpha+\gamma$. Indeed, let $\tau$ be such that
$$
	\alpha\uplus\gamma=\tau^{-1}(\tau(\alpha)+\tau(\gamma))
	\,.
$$
Any root appearing with positive coefficient in the cluster expansion of
$\tau(\alpha\uplus\gamma)$ is compatible with  $\tau(\alpha)$, $\tau(\gamma)$
and with any root compatible with both. Since $\tau$ preserves the compatibility
degree we get \ref{prop:bipartite-inductive-property_2:compatibility}. Similarly
for \ref{prop:bipartite-inductive-property_2:inequality}:
\begin{align*}
	F_{\ge-1}(\tau(\alpha))+F_{\ge-1}(\tau(\gamma))>
	F_{\ge-1}\left( \tau(\alpha)+\tau(\gamma )\right)=
	F_{\ge-1}(\tau(\alpha\uplus\gamma))
\end{align*}
since $F_{\ge-1}$ is invariant under the action of $\tau$ we can conclude
\begin{align*}
	F_{\ge-1}(\alpha)+F_{\ge-1}(\gamma)>
	F_{\ge-1}(\alpha\uplus\gamma)
	\,.
\end{align*}

We need to translate the above results to $\Phi_{\ap}(t)$; in order to do so we
need a bijection between $\Phi_{\ge-1}$ and $\Phi_{\ap}(t)$ induced by a linear
map. Note that Proposition \ref{prop:translation-pi-phi} together with Lemma 5.2
in \cite{shih} already provide a bijection but it is not induced by a linear
map.

Note also that, for a bipartite Coxeter element $t=t_+t_-$, the negative roots in
$\Phi_{\ap}(t)$ are the roots $-\beta_i^t$ given by
\begin{equation*}
	-\beta_i^t=
	\left\{ 
	\begin{array}{ll}
		-\alpha_i & i\in I_-\\
		-t_-\alpha_i & i\in I_+
		\,.
	\end{array}
	\right.
\end{equation*}

\begin{proposition}
  The linear involution
	\begin{equation*}
		t_-:Q_\mathbb{R}\longrightarrow Q_\mathbb{R}
	\end{equation*}
  restricts to an automorphism of $Q$ and to a bijection
	\begin{equation*}
		t_-:\Phi_{\ge-1}\longrightarrow\Phi_{\ap}(t)
		\,.
	\end{equation*}
\end{proposition}
\begin{proof}
	It suffices to show that, for any root $\alpha$ in $\Phi_{\ge-1}$, we have
	$t_-(\alpha)\in\Phi_{\ap}(t)$. Let us first deal with roots whose image is
	negative.
	\begin{itemize}
		\item 
			If $i\in I_+$, then 
  		$$
  			t_-(-\alpha_i)=-t_-\alpha_i=-\beta_i^t
				\,.
  		$$

		\item
			If $i\in I_-$, then
  		$$
  			t_-(\alpha_i)=s_i\alpha_i=-\alpha_i=-\beta_i^t
				\,.
  		$$
	\end{itemize}
	This also shows that $t_-(-\alpha_i)=\beta_i^t$ if $i\in I_-$. For any other
	root $\alpha$ in $\Phi_{\ge-1}$, that is, for any positive root not in
	$\left\{ \alpha_i \right\}_{i\in I_-}$, the image $t_-(\alpha)$ is positive
	since the support of any positive root is a connected sub-diagram of the
	Dynkin diagram and $t_-$ sends any root to itself plus a linear combination of
	simple roots indexed by $I_-$.

\end{proof}

\begin{proposition}
	The map $t_-$ intertwines $\tau_t$ with $\tau_-\tau_+$ and preserves
	compatibility degree. In other words, for any almost positive roots $\alpha$
	and $\gamma$, we have
	\begin{equation*}
		\tau_t(t_-\alpha)=t_-\tau_-\tau_+(\alpha)
	\end{equation*}
  and
	\begin{equation*}
		\left( \alpha||\gamma \right)_{\ge-1}=
  	\left( t_-\alpha||t_-\gamma \right)_t
	\end{equation*}
\end{proposition}
\begin{proof}
  Proceed by direct inspection;
	\begin{itemize}
		\item 
			if $t_-\alpha=\alpha_i=\beta_i^t$ for $i\in I_-$, that is if
			$\alpha=-\alpha_i$ with $i\in I_-$, then
  		$$
  			\tau_t(t_-(-\alpha_i))=
 			  \tau_t(\beta_i^t)=
			  -\beta_i^t=
 				-\alpha_i=
			  t_-\alpha_i=
			  t_-\tau_-\tau_+(-\alpha_i)
		  $$
			
		\item
			if $t_-\alpha=\beta_i^t$ for $i\in I_+$, i.e. if $\alpha=\alpha_i$ with
			$i\in I_+$, then
  		$$
  			\tau_t(t_-\alpha_i)=
  			\tau_t(\beta_i^t)=
			  -\beta_i^t=
			  -t_-\alpha_i=
			  t_-(-\alpha_i)=
			  t_-\tau_-\tau_+(\alpha_i)
		  $$

		\item
			in any other case 
  		$$
  			\tau_t(t_-\alpha)=
			  tt_-\alpha=t_+\alpha=
			  t_-t_-t_+\alpha=
			  t_-\tau_-\tau_+(\alpha)
				\,.
		  $$

	\end{itemize}
  To conclude the proof it is enough to show that 
  $$
  	\left(-\alpha_i||\gamma  \right)_{\ge-1}=
  	\left( -t_-\alpha_i||t_-\gamma \right)_t
  $$
	for any $\gamma\in\Phi_{\ge-1}$ and any $i\in I$. If $i$ is in $I_+$ then
  $$
  	\left( -\alpha_i||\gamma \right)_{\ge-1}=
  	\left[ \gamma;\alpha_i \right]_+=
  	\left[ t_-\gamma;\alpha_i \right]_+=
  	\left( -\beta_i^t||t_-\gamma \right)_t=
  	\left( t_-(-\alpha_i)||t_-\gamma \right)_t
  $$
	where the second equality holds because $t_-$ does not contain $s_i$.  If
	$i\in I_-$ then, on the one hand we have 
  $
  	\left( -\alpha_i||\gamma \right)_{\ge-1}=
  	\left[ \gamma;\alpha_i \right]_+
  $
  on the other
	\begin{align*}
  	\left( t_-(-\alpha_i)||t_-\gamma \right)_t=
		\left( \alpha_i||t_-\gamma \right)_t=
  	\left( \tau_t\alpha_i||\tau_tt_-\gamma \right)_t=
  	\left( -\beta_i^t||\tau_tt_-\gamma \right)_t=
  	\left[ \tau_tt_-\gamma;\alpha_i \right]_+.
	\end{align*}
  Now there are three cases:
	\begin{enumerate}
		\item 
  		if $\gamma$ is $\alpha_j$ with $j\in I_+$ then
  		$
				\tau_tt_-\gamma=
				\tau_t(\beta_j^t)=
				-\beta_j^t
			$
  		and
  		$$
  			\left[ \alpha_j;\alpha_i \right]_+=
  			0=
  			\left[-\beta_j^t;\alpha_i \right]_+
				\,.
  		$$

		\item
  		If $\gamma$ is $-\alpha_j$ with $j\in I_-$ then
  		$$
				\tau_tt_-\gamma=
				\tau_t(\beta_j^t)=
				-\beta_j^t=
				-\alpha_j=
				\gamma
				\,.
		  $$

		\item
  		For any other $\gamma$ we have  
  		$
  			\tau_tt_-\gamma=
				tt_-\gamma=
				t_+\gamma
  		$
  		and
  		$$
  			\left[ t_+\gamma;\alpha_i \right]_+=
  			\left[ \gamma;\alpha_i \right]_+
  		$$
  		since $s_i$ does not appear in $t_+$.

	\end{enumerate}

\end{proof}

Since the map $t_-$ is linear, all the properties of $\Phi_{\ge-1}$ translate to
$\Phi_{\ap}(t)$:
\begin{corollary}$\,$
	\label{cor:bipartite-summary_1}
\begin{enumerate}
  \item 
    For any $\alpha$ and $\gamma$ in $\Phi_{\ap}(t)$ we have $\left(
    \alpha||\gamma \right)_t=0$ if and only if $\left( \gamma||\alpha
    \right)_t=0$.
		\label{cor:bipartite-symmetry}

  \item 
		For any $J\subset I$ and any pair of roots $\alpha$ and $\gamma$ in
		$\Phi^J_{\ap}(t_J)$
    $$
    	\left( \iota \alpha||\iota \gamma \right)_t=
    	\left( \alpha||\gamma \right)^J_{t_J}
			\,.
    $$
    \label{cor:bipartite-inductive-property}

	\item 
		Each $t$-cluster in the simplicial complex $\Delta^\Phi_t$ is a
		$\mathbb{Z}$-basis of the root lattice $Q$.  

	\item
		Any $\gamma\in Q$ admits a unique $t$-cluster expansion. That is,
		$\gamma$ can be uniquely written as
		$$
			\gamma=\sum_{\alpha\in\Phi_{\ap}(t)}m_\alpha\alpha
		$$
		so that all the coefficients $m_\alpha$ are non negative integers and $m_\alpha
		m_{\alpha'}=0$ whenever $\left(\alpha||\alpha' \right)_t\neq0$.
		\label{cor:bipartite-cluster-expansion}
  \item  
		The set $\mathcal{F}_t$ is a complete simplicial fan in $Q_\mathbb{R}$.
\end{enumerate}
\end{corollary}
\begin{proof}
	The only non trivial claim is (\ref{cor:bipartite-inductive-property}); it is
	enough to show that it holds when $I\setminus J=\left\{ j \right\}$. Since $t$
	is bipartite, $s_j$ is either initial or final (cf.
	(\ref{eqn:bipartite-coxeter})). Using Proposition
	\ref{prop:inverse-invariance} we can assume it is initial, i.e.  $j\in I_+$.
	We have then
  $$
  	\left( \iota \alpha||\iota \gamma \right)_t=
  	\left( t_-\iota \alpha||t_-\iota \gamma  \right)_{\ge-1}=
  	\left( t_-\alpha||t_-\gamma \right)_{\ge-1}^J=
  	\left( \alpha||\gamma \right)^J_{t_J}
  $$
	where the second equality holds since, for any root $\alpha$ in
	$\Phi_{\ap}^J(t_J)$,
	$$
		\left[ t_-\iota(\alpha);\alpha_j \right]=0
	$$
	hence $t_-\iota(\alpha)$ is in $\Phi^J$.  Indeed if $\alpha$ is positive then
	$t_-\alpha$ contains $\alpha_j$ only if $\alpha$ does since $s_j$ does not
	appear in $t_-$; if $\alpha=-\beta_i^{t_J}$ with $i\in I_-$ then
	$t_-\iota(-\beta_i^{t_J})=\alpha_i$ and finally if $\alpha=-\beta_i^{t_J}$
	with $i\in I_+\setminus\left\{ j \right\}$ then
	$t_-\iota(-\beta_i^{t_J})=-\alpha_i$.

\end{proof}
\begin{corollary}
  \label{cor:bipartite-summary_2}
	Suppose that $I$ has at least $2$ vertices.  If $\alpha$ and $\gamma$ in
	$\Phi_{\ap}(t)$ are  such that $\left( \alpha||\gamma \right)_t=1=\left(
	\gamma||\alpha	\right)_t$, then 
	\begin{enumerate}
		\item 
			The set 
			$$
				\left\{ \tau_t^m\left( \tau_t^{-m}(\alpha)+\tau_t^{-m}(\gamma) \right)
				\right\}_{m\in\mathbb{Z}}
			$$
			consists of exactly two elements $\alpha+\gamma$ and
			$\alpha\uplus_t\gamma$.
			\label{cor:bipartite-2-elements}

		\item
			Any root appearing with a positive coefficient in the cluster expansion of $\alpha+\gamma$ or
			$\alpha\uplus_t\gamma$ is compatible with both $\alpha$, $\gamma$, and with
			any other root compatible with both $\alpha$ and $\gamma$.
			\label{cor:bipartite-compatible-expansion}

		\item
			Let $f:I\rightarrow\mathbb{R}_+$ be any function such that, for any
			$i\in I$,
			$$
				f(i)=f(i^*)
			$$
			and
			$$
				\sum_{i\in I}a_{ij}f(i)>0
			$$
			for any $j\in I$. Let $F_t:Q_\mathbb{R}\rightarrow\mathbb{R}$ be a
			contiuous piecewise-linear function on $Q_\mathbb{R}$ that is linear on
			the maximal cones of $\mathcal{F}_t$, invariant under the action of
			$\tau_t$, and such that 
			$$
				F_t(-\beta_i^t)=f(i)
				\,.
			$$
			Then,
			$$
				F_t(\alpha)+F_t(\gamma)>\max\left\{
				F_t(\alpha+\gamma),F_t(\alpha\uplus\gamma) \right\}
				\,.
			$$
			\label{cor:bipartite-inequalities}
	\end{enumerate}
\end{corollary}

\section{Some technical results}
\label{sec:technical-results}
As anticipated we need to lift elementary moves to the level of $\Phi_{\ap}(c)$.
We concentrate first on conjugation by initial simple reflections.  Fix the
Coxeter element $c=s_1\cdots s_n$ and consider the bijection
\begin{equation*}
	\sigma_1:\Phi_{\ap}(c)\longrightarrow\Phi_{\ap}(s_1cs_1)
\end{equation*}
defined by
\begin{equation}
	\sigma_1(\alpha):=
	\left\{ 
	\begin{array}[]{ll}
		\alpha_1(=\beta_1^{s_1cs_1}) & \mbox{if } \alpha=-\beta_1^c\\
		s_1\alpha & \mbox{otherwise.}
	\end{array}
	\right.
\end{equation}
Note that $\sigma_1$ sends $-\beta_i^c$ to $-\beta_i^{s_1cs_1}$ for any
$i\neq1$.

\begin{proposition}
	The map $\sigma_1$ intertwines $\tau_c$ and $\tau_{s_1cs_1}$, i.e., for any
	$\alpha$ in $\Phi_{\ap}(c)$, we have
	\begin{equation*}
		\tau_{s_1cs_1}\left( \sigma_1(\alpha) \right)=
		\sigma_1\left( \tau_c(\alpha) \right)
		\,.
	\end{equation*}
  Moreover it preserves the compatibility degree, i.e.
  for any $\alpha$ and $\gamma$ in $\Phi_{\ap}(c)$
	\begin{equation*}
		\left( \alpha||\gamma \right)_c=
		\left( \sigma_1\alpha||\sigma_1\gamma \right)_{s_1cs_1}.
	\end{equation*}
\end{proposition}
\begin{proof}
	It suffices to notice that $\sigma_1$ is the composition 
	\begin{equation*}
		\Phi_{\ap}(c)
		\stackrel{\phi_c^{-1}}{\longrightarrow}
		\Pi(c)
		\stackrel{\psi_{s_1cs_1,c}^{-1}}{\longrightarrow}
		\Pi(s_1cs_1)
		\stackrel{\phi_{s_1cs_1}}{\longrightarrow}
		\Phi_{\ap}(s_1cs_1)
	\end{equation*}
	where $\psi_{s_1cs_1,c}^{-1}$ is the bijection 
	\begin{equation*}
		\psi_{s_1cs_1,c}^{-1}(\lambda):=
		\left\{
		\begin{array}{ll}
			-\omega_1 & \mbox{if }\lambda=\omega_1\\
			s_1\lambda & \mbox{otherwise}
		\end{array}
		\right.
	\end{equation*}
	defined by Lemma 5.3 in \cite{shih} and $\phi_c$ is the map of Proposition
	\ref{prop:translation-pi-phi}.  Indeed, if $\alpha\neq\alpha_1$ then
	\begin{align*}
		\phi_{s_1cs_1}\circ\psi_{s_1cs_1,c}^{-1}\circ\phi_c^{-1}\left(\alpha\right) =
		(s_1c^{-1}s_1-1)s_1(c^{-1}-1)^{-1}\alpha=
		s_1\alpha
	\end{align*}
	and 
	\begin{align*}
		\phi_{s_1cs_1}\circ\psi_{s_1cs_1,c}^{-1}\circ\phi_c^{-1}\left(\alpha_1\right) =
		(s_1c^{-1}s_1-1)(-\omega_1)=
		\alpha_1
		\,.
	\end{align*}

	The bijection $\sigma_1$ satisfies the desired property because all
	the maps that define it do.

\end{proof}

To use simultaneously induction on the rank of $I$ and 
elementary moves we need to prove some type of compatibility between $\sigma_1$
and $\iota$.  It suffices to inspect their interaction in the case when $\iota$
is induced removing only one node (say $i$) from $I$. 

\begin{proposition}
	For $J=I\setminus\left\{ i\right\}$ and $i\neq1$ let $c_J$ be the Coxeter
	element of $W_J$ obtained by deleting $s_i$ from $c$ and let  $\sigma_1^J$ be
	the map corresponding to the conjugation by $s_1$ in $W_J$. For any root
	$\alpha$ in $\Phi^J_{\ap}(c_J)$ we have
	\begin{equation*}
		\sigma_1(\iota_c(\alpha))=\iota_{s_1cs_1}(\sigma^J_1(\alpha))
		\,.
	\end{equation*}
\end{proposition}
\begin{proof}
  There are three cases to be considered.
  \begin{enumerate}
    \item 
      If $\alpha=-\beta_1^{c_J}$ then
      $$
        \sigma_1\left( \iota_c(-\beta_1^{c_J}) \right)=
        \sigma_1\left( -\beta_1^c \right)=
        \alpha_1=
				\iota_{s_1cs_1}(\alpha_1)=
				\iota_{s_1cs_1}\left( \sigma^J_1(-\beta_1^{c_J}) \right)
				\,.
      $$

    \item
			If $\alpha=-\beta_j^{c_J}$ and $j\neq1$ then
			$\iota_c(-\beta_j^{c_J})=-\beta_j^c$ therefore
			\begin{align*}
        \sigma_1\left( \iota_c(-\beta_j^{c_J}) \right)=
				-\beta_j^{s_1cs_1}=
				\iota_{s_1cs_1}\left( -\beta_j^{s_1c_J s_1} \right)=
				\iota_{s_1cs_1}\left( \sigma^J_1(-\beta_j^{c_J}) \right)
				\,.
			\end{align*}

    \item 
      If $\alpha$ is positive
      $$
        \sigma_1\left( \iota_c(\alpha) \right)=
				\sigma_1(\alpha)=
				s_1\alpha=
				\iota_{s_1cs_1}\left( s_1\alpha \right)=
				\iota_{s_1cs_1}\left( \sigma^J_1(\alpha) \right)
				\,.
      $$
			The last equality holds since, $\alpha$ being positive,
			$s_1\alpha\neq\alpha_1$ and the third because if $s_1\alpha$ is not
			positive then $\alpha=\alpha_1$ and $-\beta_1^{s_1cs_1}=-\beta_1^{s_1c_J
			s_1}=-\alpha_1$.
        
  \end{enumerate}

\end{proof}

\begin{remark}
	The definition of $\sigma_1$ can be replicated to get the maps $\sigma_i$
	corresponding to conjugation by any initial simple reflection $s_i$.  It is
	clear that, to get the maps corresponding to elementary moves that conjugate
	$c$ by a final simple reflection, it suffices to consider the inverses 
	$\sigma_i^{-1}$. 
\end{remark}

As a  first application of the elementary moves let us show that the definition
of $c$-compatible pair of roots make sense.

\begin{lemma}
	For any $\alpha$ and $\gamma$ in $\Phi_{\ap}(c)$
	\begin{equation*}
		\left( \alpha||\gamma \right)_c=0
		\Leftrightarrow
		\left( \gamma||\alpha \right)_c=0
		\,.
	\end{equation*}
	\label{lemma:symmetry}
\end{lemma}
\begin{proof}
	If $c$ is bipartite the statement is true by point
	\ref{cor:bipartite-symmetry} in Corollary \ref{cor:bipartite-summary_1}.  It
	is then  enough to show that the property is preserved under
	$\sigma_i^{\pm1}$. Suppose that it holds for $c=s_1\cdots s_n$.  
	If $\left(\alpha||\gamma \right)_{s_1cs_1}=0$ then 
	$$
		\left(\sigma_1^{-1}\alpha||\sigma_1^{-1}\gamma \right)_c=
		0=
		\left( \sigma_1^{-1}\gamma||\sigma_1^{-1}\alpha \right)_c
		\,.
	$$
	Therefore	$\left( \gamma || \alpha\right)_{s_1cs_1}=0$.

\end{proof}

Our next goal is to show that ``distant'' roots are compatible.  We need to
introduce some terminology.  For any positive root $\alpha$ define its support
to be the set
\begin{equation}
	\supp(\alpha):=\left\{ i\in I \Big| \left[ \alpha;\alpha_i \right]\neq 0 \right\}
\end{equation}
and extend the definition to $\Phi_{\ap}(c)$ declaring
\begin{equation}
	\supp(-\beta_i^c):=\left\{ i \right\}
	\,.
\end{equation}

\begin{remark}
	If $\alpha$ and $\gamma$ are roots with supports contained in two different
	connected components of $I$ then $\left( \alpha||\gamma \right)_c=0$  since
	$\tau_c$ preserves connected components. 
	\label{rk:connected-component-compatibility}
\end{remark}

We can improve on Remark \ref{rk:connected-component-compatibility}.
\begin{definition}
	Call two roots $\alpha$ and $\gamma$ \emph{spaced} if, for any
	$i\in\supp(\alpha)$ and for any $j\in\supp(\gamma)$, $a_{ij}=0$. 
\end{definition}

\begin{remark}
	Note that if $\left( -\beta_i^c||\alpha \right)_c\neq0$ then $\alpha$ and
	$-\beta_i^c$ are not spaced.
	\label{rk:compatibility-of-negative}
\end{remark}

\begin{proposition}
	Let $\alpha$ and $\gamma$ be roots in $\Phi_{\ap}(c)$. If $\alpha$ and
	$\gamma$ are spaced then 
  $$
		\left( \alpha||\gamma \right)_c=0.
	$$
	\label{prop:spaced-compatibility}
\end{proposition}
\begin{proof}
	Using Remark \ref{rk:connected-component-compatibility} we can assume that $I$
	is connected.  If any of $\alpha$ and $\gamma$ is a negative root we are done
	by Lemma \ref{lemma:symmetry} and Remark \ref{rk:compatibility-of-negative}.
	Let then both $\alpha$ and $\gamma$ be positive roots.

	Supports of positive roots are connected subgraphs of the Dynkin diagram $I$.
	Since $\alpha$ and $\gamma$ are spaced, there must exist at least one vertex
	on the shortest path connecting $\supp(\alpha)$ and $\supp(\gamma)$ not
	belonging to either of the supports. Let $i$ be the nearest to $\supp(\alpha)$
	of such vertices.  Let $I'$ be a connected component of $I\setminus\left\{ i
	\right\}$ of type $A$ and containing one of the two support; there exist such
	a component because we are in finite type.  Assume $\alpha$ is the root whose
	support is contained in $I'$ (the other case is identical). We will proceed by
	induction on the cardinality of $I'$.
  
	Let $j$ be the only vertex in $\supp(\alpha)$ connected to $i$.  Without loss
	of generality we can assume $j\prec_c i$, i.e., $s_j$ precedes $s_i$ in any
	reduced expression of $c$.  If this is not the case  we can use Proposition
	\ref{prop:inverse-invariance} since two roots are spaced if and only if their
	images under the involution $\delta\mapsto\overline\delta$ are spaced.
  
	Apply $\tau_c^{-1}$ to both $\alpha$ and $\gamma$; they are  positive so
	$\tau_c^{-1}$ acts as $c^{-1}$ on them. By construction we have
	$$
		\supp\left( \tau_c^{-1}\alpha \right)\subseteq
		I'\setminus\left\{ j	\right\}
	$$
	and
	$$
		\supp\left( \tau_c^{-1}\gamma \right)\subseteq
		\left( I\setminus I' \right)\cup\left\{ i \right\}
	$$
	where both relations hold since $s_j$ is applied before $s_i$ and $\alpha$
	belongs to a type $A$ component of $I$ .  If one among $\tau_c^{-1}\alpha$ and
	$\tau_c^{-1}\gamma$ is negative we are done (again using Lemma
	\ref{lemma:symmetry} if needed) otherwise the statement follows by induction
	on $|I'|$.

\end{proof}

To complete the proof of Proposition \ref{prop:inductive-property} we need to
sharpen Lemma \ref{lemma:coxeter-connected}.  From this moment on we will denote
a word on the alphabet $\left\{ s_i \right\}_{i\in I}$ (up to commutations) by
$\mathbf{w}$; the corresponding element in $W$ will be denoted by $w$.  For
convenience we will record a sequence of elementary moves by the corresponding
word. As an example in type $A_4$ (again using the standard numeration of simple
roots from \cite{bourbaki}) the sequence of elementary moves
$$
	s_1s_2s_3s_4\rightarrow s_2s_3s_4s_1\rightarrow s_1s_3s_4s_2
$$
will be encoded by $\mathbf{w}=s_2s_1$; indeed 
$$
	(s_2s_1)(s_1s_2s_3s_4)(s_2s_1)^{-1}=s_1s_3s_4s_2
	\,.
$$

The key observation is given by the following Lemma.
\begin{lemma}
	For any pair of Coxeter elements $c$ and $c'$ and for any $i\in I$, there
	exist a sequence of elementary moves connecting $c$ and $c'$ that does not
	contain $s_i$.
	\label{lemma:w-without-i}
\end{lemma}
\begin{proof}
	The result is obvious once we notice that both the sequences of simple moves
	$\mathbf{w}=s_1$ and $\mathbf{w'}=s_2\cdots s_n$ acts in the same way on
	$c=s_1\cdots s_n$.

\end{proof}

The Dynkin diagram $I$ is in general a forest. For any leaf $i$ in $I$, i.e.,
for any node belonging to a single edge, denote by $i_\#$ the only other node of
$I$ connected to $i$.

\begin{lemma}
	For any leaf $i\in I$ and for any Coxeter element $c$ there exist a bipartite
	Coxeter element $t$ and a sequence of elementary moves $\mathbf{w}$ such that 
  \begin{enumerate}
    \item $c=\mathbf{w}t\mathbf{w}^{-1}$
    \item $\mathbf{w}$ contains neither $s_i$, nor $s_{i_\#}$.
  \end{enumerate}
  \label{lemma:distant-conjugation}
\end{lemma}
\begin{proof}
	According to Lemma \ref{lemma:w-without-i} we can find a sequence of
	elementary moves $\mathbf{w}$ not containing $s_{i_\#}$ that transform $c$
	into either of the bipartite Coxeter elements.  By construction $s_i$ commutes
	with all reflections appearing in $\mathbf{w}$ since it commutes with all the
	simple reflections except $s_{i_\#}$.  Therefore, using commutations
	relations, we can always find such a $\mathbf{w}$ containing at most one copy
	of $s_i$.  Choosing now a bipartite Coxeter element $t$ in which $s_i$ and
	$s_{i_\#}$ appear in the same order in which they appear in $c$ we have that
	$\mathbf{w}$ does not contain $s_i$.

\end{proof}

Let $i$ be a leaf of $I$  and $c$ be any Coxeter element.  Let $t$ and
$\mathbf{w}$ be respectively the bipartite Coxeter element and the sequence of
elementary moves constructed in Lemma \ref{lemma:distant-conjugation}. To fix
ideas suppose that $s_i$ appears on the left of $s_{i_\#}$ in $c$ (and in $t$);
the other case can be dealt with exactly in the same way but multiplying on the
right instead of on the left.  Denote by $c_J$ and $t_J$ the corresponding
Coxeter elements for the Dynkin sub diagram $J=I\setminus\left\{ i \right\}$. By
our assumption $c_J=s_ic$ and $t_J=s_it$.
\begin{lemma}
	In the notation just established $\mathbf{w}$ is a sequence of elementary
	moves in $W_J$ conjugating  $t_J$ and $c_J$.
  \label{lemma:respect-conjugation}
\end{lemma}
\begin{proof}
  $$
		c_J=
		s_ic=
		s_i\mathbf{w}t\mathbf{w}^{-1}=
		\mathbf{w}s_it\mathbf{w}^{-1}=
		\mathbf{w}t_J\mathbf{w}^{-1}.
	$$

\end{proof}

We now have all the required tools to prove Proposition
\ref{prop:inductive-property}.
\begin{proof}[Proof of Proposition \ref{prop:inductive-property}]
	We can assume, without loss of generality, $I$ to be irreducible.  It suffices
	to show that the result holds when $J$ is obtained from $I$ removing one node
	$i$.  Let $\alpha$ and $\gamma$ be roots in $\Phi_{\ap}(c_J)$.  There are two
	cases to consider depending on the relative position of $\supp(\alpha)$,
	$\supp(\gamma)$ and $i$. 
  \begin{enumerate}
    \item 
			If $\supp(\alpha)$ and $\supp(\gamma)$ belong to different connected
			components of $J$ then 
      $$
        \left( \alpha||\gamma \right)^J_{c_J}=
				0=
				\left( \iota(\alpha)||\iota(\gamma) \right)_c.
      $$
			The first equality holds because of Remark
			\ref{rk:connected-component-compatibility} and the second one is an
			instance of Proposition \ref{prop:spaced-compatibility}.

		\item
			If $\supp(\alpha)$ and $\supp(\gamma)$ belong to the same connected
			component of the Dynkin diagram $J$ then we can assume $i$ to be a leaf of
			$I$.  Let $i_\#$ be the only vertex in $I$ connected to $i$. By Lemmata
			\ref{lemma:distant-conjugation} and \ref{lemma:respect-conjugation}, there
			exist a sequence of elementary moves $\mathbf{w}$ and a bipartite Coxeter
			element $t$ such that 
			$$
				c=\mathbf{w}t\mathbf{w}^{-1}
			$$
			$$
				c_J=\mathbf{w}t_J\mathbf{w}^{-1}
			$$
			and $\mathbf{w}$ contains neither $s_i$, nor $s_{i_\#}$. Denote by
			$\sigma_{\mathbf{w}}$ the composition of the maps $\sigma_i$ corresponding
			to $\mathbf{w}$. By construction neither of
			$\sigma_{\mathbf{w}}\alpha$ and $\sigma_{\mathbf{w}}\gamma$ contains $i$ in its
			support.
			Using point \ref{cor:bipartite-inductive-property} of
			Corollary \ref{cor:bipartite-summary_1} we can conclude
			$$
				\left( \alpha||\gamma \right)_c=
				\left( \sigma_{\mathbf{w}}\alpha||\sigma_{\mathbf{w}}\gamma \right)_t=
				\left( \sigma_{\mathbf{w}}\alpha||\sigma_{\mathbf{w}}\gamma \right)^J_{t_J}=
				\left( \alpha||\gamma \right)^J_{c_J}.
			$$
  \end{enumerate}

\end{proof}

\section{Proof of the  main results}
\label{sec:proof-main-results} 
To prove Theorem \ref{thm-phi:z-basis} we will use the following easy
observation.
\begin{lemma}
	Let $J\subset I$ be a Dynkin sub diagram.  There is a bijection between
	$c$-clusters in $\Phi_{\ap}(c)$ containing $\left\{ -\beta_i^c \right\}_{i\in
	I\setminus J}$ and $c_J$-clusters in $\Phi^J_{\ap}(c_J)$.  
	\label{lemma:cluster-induction}
\end{lemma}
The existence of such a bijection is a direct consequence of the fact that, if
$-\beta_i^c$ is in a $c$-cluster $C$, then for any other root $\gamma$ in that
$c$-cluster $i\not\in\supp(\gamma)$.  In particular any positive root in $C$ is
contained in a complement of the space generated by 
$$
	\left\{-\beta_i^c  \right\}_{-\beta_i^c\in C\cap\Phi_-}.
$$

\begin{definition}
	Call a $c$-cluster $C$ in $\Phi_{\ap}(c)$  \emph{positive} if $C\subset\Phi_+$.
\end{definition} 

\begin{proof}[Proof of Theorem \ref{thm-phi:z-basis}]
	As already mentioned in Remark \ref{rk:triangular-change-of-base}, the set
	$$
		\left\{ -\beta_i^c \right\}_{i\in I}
	$$
	is a $\mathbb{Z}$-basis of $Q$.	By
	Lemma  \ref{lemma:cluster-induction} it suffices to show that the theorem
	holds for any given positive cluster $C$. Apply $\tau_c^{-1}$ to $C$; since
	it is positive $\tau_c^{-1}$ acts on all the roots in it as  $c^{-1}$.  Since
	$c^{-1}$ is a product of reflections, $C$ is a $\mathbb{Z}$-basis if and only
	if  $\tau_c^{-1}C$ is a $\mathbb{Z}$-basis. Continue to apply $\tau_c^{-1}$
	until one of the roots is sent to a negative one. This will happen because,
	similarly to the case of $\Pi(c)$, any $\tau_c$-orbit contains precisely two
	negative roots $\left\{ -\beta_i^c,-\beta_{i^*}^c \right\}$ or a single
	negative root $-\beta_i^c$ if $i=i^*$. 
	Remove the negative root just obtained  again using Lemma
	\ref{lemma:cluster-induction} and conclude by induction on the rank of the
	root system.

\end{proof}

\begin{remark}
	The proof just proposed is a straightforward adaptation of the proof of
	Theorem 1.8 in \cite{y-system} to the new setup of $c$-almost positive roots.
\end{remark}

\begin{proof}[Proof of Proposition \ref{prop-phi:unique-cluster-expansion}]
	Let $c=s_1\dots s_n$; for $\gamma\in Q$ write 
	\begin{equation*}
		\gamma=-\sum_{i\in I}m_{-\beta_i^c}\beta_i^c+\gamma_+
	\end{equation*}
	where the coefficients $m_{-\beta_i^c}$ are the non-negative integers uniquely
	defined  (since the change of basis $\alpha_i\mapsto\beta_i^c$ is triangular)
	by the recursive formula 
	\begin{equation*}
		m_{-\beta_i^c}:=
		\left[-\gamma-\sum_{j=1}^{i-1}m_{-\beta_j^c}\beta_j^c;\alpha_n  \right]_+
	\end{equation*}
	(we use the convention that the empty sum is $0$).
	By construction $\gamma_+$, the \emph{positive part} of $\gamma$, is in the
	positive cone of the sub-root lattice generated by 
	$$
		\left\{ \alpha_i | m_{-\beta_i^c}=0 \right\}.
	$$
	Clearly $\gamma\in Q$ has a unique $c$-cluster expansion if and only if
	$\gamma_+$ does. 
	Without loss of generality we can thus assume that $\gamma$ is in the positive cone
	$Q_+$.

	The root $-\beta_i^c$ can appear with a positive coefficient in a $c$-cluster
	expansion of $\gamma$ only if the coefficient $[\gamma;\alpha_i]$ is negative
	therefore the result holds for $\gamma=0$ and we can assume $\gamma\neq0$.

	If $\sum_{\alpha\in\Phi_+}m_\alpha\alpha$ is a $c$-cluster expansion of
	$\gamma$ then
	$$
		\tau_c^{-1}\gamma=
		c^{-1}\gamma=
		c^{-1}\left( \sum_{\alpha\in\Phi_+}m_\alpha\alpha	\right)=
		\sum_{\alpha\in\Phi_+}m_\alpha c^{-1}\alpha=
		\sum_{\alpha\in\Phi_+}m_\alpha\tau_{c^{-1}}\alpha
	$$
	is a $c$-cluster expansion of $\tau_c^{-1}\gamma$. In other words $\gamma$
	has a unique $c$-cluster expansion if and only if $\tau_c^{-1}\gamma$ does.
	Applying $\tau_c^{-1}$ a sufficient number of times $\gamma$ can be moved
	outside of the positive cone.  We can then take its positive part and conclude
	by induction on the rank of the root system.	

\end{proof}

\begin{remark}
	This proof, as its analog in \cite{y-system}, has the advantage of considering
	one Coxeter element at a time. An alternative strategy could have been the
	following. The claim holds for bipartite Coxeter elements by point
	\ref{cor:bipartite-cluster-expansion} in Corollary
	\ref{cor:bipartite-summary_1}. Using the fact that the maps $\sigma_i$
	preserve compatibility degree one can then transfer the property to other sets
	of $c$-almost positive roots.
\end{remark}

As in \cite{y-system}, Theorem \ref{thm-phi:complete-fan} follows from
Proposition
\ref{prop-phi:unique-cluster-expansion}. For the sake of completeness we replicate
the proof here.
\begin{proof}[Proof of Theorem \ref{thm-phi:complete-fan}]
	It suffices to show that
	\begin{enumerate}
		\item 
			no two cones of $\mathcal{F}_c^\Phi$ have a common interior point
		\item
			the union of all cones is $Q_{\mathbb{R}}$.
	\end{enumerate}
	Assume by contradiction that there exist a point in the common interior of two
	cones. Since $Q_{\mathbb{Q}}$ is dense in $Q_{\mathbb{R}}$ we may assume that
	such point is in $Q_{\mathbb{Q}}$. Clearing the denominators there is then a
	common point in $Q$ which contradicts the uniqueness of the $c$-cluster
	expansion. Therefore the interiors of any two cones are disjoint. 
	
	Since for any $\gamma\in Q$ there exist a $c$-cluster expansion the union of
	all the cones $\mathbb{R}_+C$ contains $Q$; since this union is closed in
	$Q_{\mathbb{R}}$ and stable under the action of $\mathbb{R}_+$ it must contain
	all of $Q_\mathbb{R}$ and we are done.

\end{proof}

To prove Theorem \ref{thm-phi:polytopality} we will apply the criterion provided
by Lemma 2.1 in \cite{associahedra}. Let us restate it in the particular case we
need.
\begin{lemma}
	Let $F_c$ be a continuous piecewise-linear function 
	$$
		F_c:Q_{\mathbb{R}}\longrightarrow\mathbb{R}
	$$
	linear on the maximal cones of the fan $\mathcal{F}_c^\Phi$ (as such $F_c$ is
	uniquely determined by its values on $\Phi_{\ap}(c)$). Then
	$\mathcal{F}_c^\Phi$ is the normal fan to a unique full-dimensional polytope
	with support function $F_c$ if and only if $F_c$ satisfy the following system
	of inequalities.  For any pair of adjacent $c$-clusters $C_\alpha$ and
	$C_\gamma$ let $\alpha$ be the only root in $C_\alpha\setminus C_\gamma$ and
	$\gamma$ the only root in $C_\gamma\setminus C_\alpha$. Let
	\begin{align}
		m_\alpha\alpha+m_{\gamma}\gamma=
		\sum_{\delta\in C_\alpha\cap C_\gamma}m_\delta \delta
		\label{linear-dependence}
	\end{align}
	be the unique (up to non-zero scalar multiple) linear dependence on the
	elements of $C_\alpha\cup C_\gamma$ with $m_\alpha$ and $m_{\gamma}$ positive.
	Then
	\begin{equation}
		m_\alpha F_c(\alpha)+m_{\gamma}F_c(\gamma)>\sum_{\delta\in C_\alpha\cap C_\gamma}m_\delta
		F_c(\delta)
		\,.
	\end{equation}
	In particular the domains of linearity of $F_c$ are exactly the maximal cones
	of $\mathcal{F}_c^\Phi$.
	\label{lemma:criterion-polytopality}
\end{lemma}

To apply Lemma \ref{lemma:criterion-polytopality} we make the relations
(\ref{linear-dependence}) more explicit by  exploring the interaction of
$\sigma_1$ and $\tau_c$.  Note that, having established Theorem
\ref{thm-phi:complete-fan}, any vector-valued function on $\Phi_{\ap}(c)$ can be
extended to a continuous piecewise-linear map on $Q_\mathbb{R}$; in particular
this is the case for $\tau_c$ and $\sigma_i$.

To avoid degenerate cases, from now on, assume that every connected component of
$I$ contains at least $2$ vertices.  As before, let $c$ be $s_1\cdots s_n$.
\begin{lemma}
	Let $\alpha$ and $\gamma$ be roots in $\Phi_{\ap}(c)$ such that
	$$
		\left( \alpha||\gamma \right)_c=
		1=
		\left( \gamma||\alpha \right)_c
		\,.
	$$
	Then
	$$
		\sigma_1^{-1}\left( \sigma_1(\alpha)+\sigma_1(\gamma) \right)
	$$ 
	is either $\alpha+\gamma$ or 
	$$
		\tau_c\left( \tau_c^{-1}(\alpha)+\tau_c^{-1}(\gamma)\right)
	$$
	and it is different from $\alpha+\gamma$ only if one of the two roots (say,
	$\alpha$) is $-\beta_1^c$; in this case
	$$
		\sigma_1^{-1}\left( \sigma_1(-\beta_1^c)+\sigma_1(\gamma) \right)=
		\gamma-\alpha_1\in Q_+
		\,.
	$$
	\label{lemma:sigma-is-tau}
\end{lemma}

\begin{proof}
	If both $\alpha$ and $\gamma$ are positive roots then
	\begin{align*}
		\sigma_1^{-1}\left( \sigma_1(\alpha)+\sigma_1(\gamma) \right)= 
		\sigma_1^{-1}\left( s_1\alpha+s_1\gamma \right)=
		\sigma_1^{-1}\left( s_1\left( \alpha+\gamma \right)\right).
	\end{align*}
	Let
	$$
    \alpha+\gamma=\sum_{\delta\in \Phi_{\ap}(c)} m_\delta\delta
  $$
	be the $c$-cluster expansion of $\alpha+\gamma$; all the roots $\delta$ such
	that $m_\delta\neq0$ are positive; therefore 
	\begin{align*}
		\sigma_1^{-1}\left( s_1\left( \alpha+\gamma \right)\right)=
		\sigma_1^{-1}\left( s_1\left( \sum m_\delta\delta \right) \right)=
		\sigma_1^{-1}\left( \sigma_1\left( \alpha+\gamma \right) \right)= 
		\alpha+\gamma.
	\end{align*}
	
	It remains to consider the case in which one of the two roots is negative
	(they cannot be both negative); we can, by symmetry, assume $\alpha$ to be the
	negative one. If $\alpha=-\beta_i^c$ with $i\neq 1$ then 
	$$
		\sigma_1^{-1}\left( \sigma_1(-\beta_i^c)+\sigma_1(\gamma) \right)=
		\sigma_1^{-1}\left( -s_1\beta_i^c+s_1\gamma \right)=
		\sigma_1^{-1}\left( s_1\left( \gamma-\beta_i^c \right) \right)
		\,.
	$$
	Let 
	$$
		\gamma-\beta_i^c=\sum_{\delta\in \Phi_{\ap}(c)} m_\delta \delta
	$$
	be the $c$-cluster expansion of $\gamma-\beta_i^c$. None of the roots $\delta$
	appearing with a positive coefficient is $-\beta_1^c$ since 
	$$
		\left[ \gamma-\beta_i^c;\alpha_1 \right]\ge0
	$$
	($-\beta_1^c$ is the only negative root having $\alpha_1$ with non-zero
	coefficient). Therefore on $ \gamma-\beta_i^c$ the actions of $s_1$ and of
	$\sigma_1$ are the same. We get
	$$
		\sigma_1^{-1}\left( \sigma_1(-\beta_i^c)+\sigma_1(\gamma) \right)=
		\sigma_1^{-1}\left( s_1\left( \gamma-\beta_i^c \right) \right)=
		\sigma_1^{-1}\left( \sigma_1\left( \gamma-\beta_i^c \right) \right)=
		\gamma-\beta_i^c
		\,.
	$$
	Finally if $\alpha=-\beta_1^c$ then on the one hand we have 
	\begin{align*}
		\sigma_1^{-1}\left( \sigma_1(-\beta_1^c)+\sigma(\gamma) \right)=
		\sigma_1^{-1}\left( \alpha_1+s_1\gamma \right)=
		\sigma_1^{-1}\left( s_1(\gamma-\alpha_1) \right)
	\end{align*}
	$\gamma-\alpha_1$ is in $Q_+$ therefore 
	\begin{align*}
		\sigma_1^{-1}\left( s_1(\gamma-\alpha_1) \right)=
		\sigma_1^{-1}\left( \sigma_1(\gamma-\alpha_1) \right)=
		\gamma-\alpha_1
	\end{align*}
	On the other hand
	\begin{align*}
		\tau_c\left( \tau_c^{-1}(-\beta_1^c)+\tau_c^{-1}(\gamma) \right)=
		\tau_c\left(\beta_1^c+c^{-1}\gamma\right)=
		\tau_c\left( c^{-1}(\gamma-\alpha_1) \right)
	\end{align*}
	we can interchange $c^{-1}$ and $\tau_c^{-1}$ because
	$\gamma-\alpha_1$ is in $Q_+$ and conclude
	\begin{align*}
		\tau_c\left( c^{-1}(\gamma-\alpha_1) \right)=
		\tau_c\left( \tau_c^{-1}(\gamma-\alpha_1) \right)=
		\gamma-\alpha_1.
	\end{align*}

\end{proof}

\begin{proposition}
	Let $\alpha$ and $\gamma$ be roots in $\Phi_{\ap}(c)$ such that
	\begin{equation*}
		\left( \alpha||\gamma \right)_c=
		1=
		\left( \gamma||\alpha \right)_c
		\,.
	\end{equation*}
	Then 
	$$
		\left\{ \tau_c^m\left( \tau_c^{-m}(\alpha)+\tau_c^{-m}(\gamma) \right)
		\right\}_{m\in\mathbb{Z}}
	$$
	consist of exactly two elements, one is $\alpha+\gamma$; denote the other by
	$\alpha\uplus_c\gamma$.
	\label{prop:2-element-orbit}
\end{proposition}

\begin{proof}
	If $c$ is bipartite there is nothing to prove by point
	\ref{cor:bipartite-2-elements} in Corollary \ref{cor:bipartite-summary_2}.  In
	view of Lemma \ref{lemma:sigma-is-tau}, for any other Coxeter element
	$c=s_1\cdots s_n$ and any integer $m$ we have
	\begin{align*}
		\tau_{s_1cs_1}^m\left(\tau_{s_1cs_1}^{-m}(\sigma_1(\alpha))+\tau_{s_1cs_1}^{-m}(\sigma_1(\gamma))\right)=
		\sigma_1\tau_c^m\sigma_1^{-1}\left(	\sigma_1\tau_c^{-m}(\alpha)+\sigma_1\tau_c^{-m}(\gamma) \right).
	\end{align*}
  For a suitable $m'$ in $\left\{ m,m+1 \right\}$ we get 
	\begin{align*}
		\sigma_1\tau_c^m\sigma_1^{-1}\left(	\sigma_1\tau_c^{-m}(\alpha)+\sigma_1\tau_c^{-m}(\gamma) \right)=
		\sigma_1\tau_c^{m'}\left( \tau_c^{-m'}(\alpha)+\tau_c^{-m'}(\gamma)\right).
	\end{align*}
  Therefore 
	\begin{equation*}
		\left\{
		\sigma_1(\alpha)+\sigma_1(\gamma),\sigma_1(\alpha)\uplus_{s_1cs_1}\sigma_1(\gamma)
		\right\}\subseteq
		\sigma_1\left\{ \alpha+\gamma,\alpha\uplus_c\gamma \right\}
	\end{equation*}
	and the claim follows reversing the role of $c$ and $s_1cs_1$.

\end{proof}

\begin{remark}
	Following Remark 1.15 in \cite{associahedra}, if $I$ contains a component of
	type $A_1$ let $\alpha_1$ and $-\alpha_1$ be the corresponding roots.  In view
	of 	Remark \ref{rk:connected-component-compatibility} they are both compatible
	with any root in $\Phi_{\ap}(c)\setminus\left\{ \alpha_1,-\alpha_1 \right\}$.
	By direct inspection, we have
	$$
		\left( -\alpha_1||\alpha_1 \right)_c=
		1=
		\left( \alpha_1 ||-\alpha_1\right)_c
	$$
	In this case their sum is $0$ and it is natural to declare
	$-\alpha_1\uplus_c\alpha_1$ to be $0$ too.
\end{remark}

\begin{corollary}
	If 
	$$
		\left( \alpha||\gamma \right)_c=
		1=
		\left( \gamma||\alpha \right)_c
	$$
	then every root appearing with positive coefficient in the cluster expansion
	of either $\alpha+\gamma$ or $\alpha\uplus_c\gamma$ is compatible with
	$\alpha$, $\gamma$ and with any other root compatible with both $\alpha$ and
	$\gamma$.
	\label{cor:c.e.-intersection}
\end{corollary}
\begin{proof}
	The statement is true in the bipartite case by point
	\ref{cor:bipartite-compatible-expansion} in Corollary
	\ref{cor:bipartite-summary_2}.  For an arbitrary Coxeter element $c=s_1\cdots
	s_n$ the result can be deduced using elementary moves: from the previous proof
	we have
	\begin{equation}
		\left\{
		\sigma_1(\alpha)+\sigma_1(\gamma),\sigma_1(\alpha)\uplus_{s_1cs_1}\alpha_1(\gamma)
		\right\}=
		\sigma_1\left( \left\{ \alpha+\gamma,\alpha\uplus_c\gamma \right\} \right)
		\label{eqn:sums-in-sums}
	\end{equation}
	and the claim follows since $\sigma_1$ preserves compatibility degrees.

\end{proof}

\begin{lemma}
	In every dependence relation (\ref{linear-dependence}) we have
	\begin{equation}
		\left( \alpha||\gamma \right)_c=
		1=
		\left( \gamma||\alpha \right)_c
		\,.
		\label{eqn:root-exchangeable}
	\end{equation}
	Futhermore, after normalization the relation (\ref{linear-dependence}) is just
	the $c$-cluster expansion of $\alpha+\gamma$:
	$$
		\alpha+\gamma=\sum_{\delta\in \Phi_{\ap}(c)}m_\delta \delta
		\,.
	$$
	\label{lemma:reduction-to-c.e.}
\end{lemma}
\begin{proof}
	Normalize (\ref{linear-dependence}) so that coefficients are coprime integers.
	By Theorem \ref{thm-phi:z-basis} all the coefficients in 
	$$
		\alpha=
		-\frac{m_\gamma}{m_\alpha}\gamma 
		+\sum_{\delta\in C_\alpha\cap C_\gamma}\frac{m_\delta}{m_\alpha} \delta
  $$
	are integers forcing $m_\alpha=1$ (it is positive by hypothesis). In a similar
	fashion $m_\gamma=1$. 

	To show (\ref{eqn:root-exchangeable}), using the $\tau_c$-invariance of the
	compatibility degree, it suffices to consider the case $\alpha=-\beta_i^c$. We
	have 
	$$
		\gamma=\beta_i^c+\sum_{\delta\in C_{-\beta_i^c}\cap C_\gamma}m_\delta \delta
	$$
	and thus 
	$$
		\left( -\beta_i^c||\gamma \right)_c=\left[ \gamma;\alpha_i \right]_+=1
	$$
	since $i$ is not in $\supp(\delta)$ for any $\delta$ in $C_{-\beta_i^c}\cap
	C_\gamma$.

	The fact that, after the normalization, the dependence
	(\ref{linear-dependence}) is the $c$-cluster expansion of $\alpha+\gamma$ is a
	direct application of Corollary \ref{cor:c.e.-intersection}. Any root
	appearing with non zero coefficient  in the cluster expansion of
	$\alpha+\gamma$ is compatible with $\alpha$, $\gamma$, and with any other root
	compatible with both $\alpha$ and $\gamma$, therefore it is a root in
	$C_\alpha \cap C_\gamma$.

\end{proof}

Proposition \ref{prop:2-element-orbit} together with Corollary
\ref{cor:c.e.-intersection} and Lemma \ref{lemma:reduction-to-c.e.} allow us to
compute exchange relations. Let $\mathcal{A}_0(c)$ be the coefficient-free
cluster algebra with initial exchange matrix $B(c)$ and denote by $\left\{
x_{\alpha,c} \right\}_{\alpha\in\Phi_{\ap}(c)}$ its cluster variables.  Due to
Proposition \ref{prop-phi:unique-cluster-expansion} all the cluster monomials
are in bijection with points of $Q$. Namely we can write
$$
	x_{\gamma,c}:=\prod_{\delta\in \Phi_{\ap}(c)} x_{\delta,c}^{m_\delta}
$$
where 
$$
	\gamma=\sum_{\delta\in \Phi_{\ap}(c)} m_\delta \delta
$$
is the cluster expansion of $\gamma\in Q$.
\begin{proof}[Proof of Theorem \ref{thm-phi:exchange-relations}]
	The statement is true when $c$ is a bipartite Coxeter element (cf. (5.1) in
	\cite{ca2}). Let $c=s_1\cdots s_n$.  We have
	\begin{align*}
		x_{\alpha,s_1cs_1}x_{\gamma,s_1cs_1}
		&
		=x_{\sigma_1^{-1}(\alpha),c}x_{\sigma_1^{-1}(\gamma),c}\\
		&=x_{\sigma_1^{-1}(\alpha)+\sigma_1^{-1}(\gamma),c}+x_{\sigma_1^{-1}(\alpha)\uplus_c\sigma_1^{-1}(\gamma),c}\\
		&=x_{\sigma_1(\sigma_1^{-1}(\alpha)+\sigma_1^{-1}(\gamma)),s_1cs_1}+x_{\sigma_1(\sigma_1^{-1}(\alpha)\uplus_c\sigma_1^{-1}(\gamma)),s_1cs_1}\\
		&=x_{\alpha+\gamma,s_1cs_1}+x_{\alpha\uplus_{s_1cs_1}\gamma,s_1cs_1}.
	\end{align*}

\end{proof}

Recall Remark \ref{rk:wheels}: by construction of the map $\tau_c$, there is one
$\tau_c$-orbit in $\Phi_{\ap}(c)$ for each $w_0$-orbit in $I$, i.e., there exist
$-\beta_j^c$ such that 
$$
	\tau_c^m(-\beta_i^c)=-\beta_j^c
$$
if and only if $j\in\left\{ i,i^* \right\}$.

Since $\sigma_j$ sends $-\beta_i^c$ to $\left\{\pm\beta_i^{s_jcs_j}\right\}$ the
$\tau_c$-orbit of $-\beta_i^c$ gets mapped to the $\tau_{s_jcs_j}$-orbit of
$-\beta_i^{s_jcs_j}$. In particular, for any function
$$
	f:I\longrightarrow \mathbb{R}
$$
such that
$$
	f(i)=f(i^*)
$$
we get a family of maps, one for each Coxeter element $c$, 
$$
	F_c=F_{c;f}:\Phi_{\ap}(c)\longrightarrow \mathbb{R}
$$
defined setting $F_c(-\beta_i^c):=f(i)$ and extending by $\tau_c$-invariance.
These maps are invariant under the action of $\sigma_i$, that is
\begin{equation} 
	F_{s_ics_i}(\sigma_i(\alpha))=F_c(\alpha)
	\label{eqn:F-invariance} 
\end{equation} 
for any $c$, any $i$ initial in $c$, and any $\alpha$ in $\Phi_{\ap}(c)$.  
From now on assume that $F_c$ has been defined
in this way and extend it to a continuous, piecewise-linear function
$$
	F_c:Q_\mathbb{R}\longrightarrow\mathbb{R}
$$ 
linear on maximal cones of $\mathcal{F}_c$.

\begin{proposition}
	Fix any function 
	$$
		f:I\longrightarrow\mathbb{R}
	$$
	such that
	\begin{enumerate}
		\item 
			for any $i\in I$
			$$
				f(i)=f(i^*)
			$$

		\item
			for any $j\in I$
			$$
				\sum_{i\in I}a_{ij}f(i)>0
				\,.
			$$
	\end{enumerate}
	Then for any pair of roots $\alpha$ and $\gamma$ in $\Phi_{\ap}(c)$ such that
	$$
		\left( \alpha||\gamma \right)_c=
		1=
		\left( \gamma||\alpha \right)_c
	$$
	the following inequality holds
	$$
		F_c(\alpha)+F_c(\gamma)>
		\max\left\{F_c(\alpha+\gamma),F_c(\alpha\uplus_c\gamma)\right\}
		\,.
	$$
  \label{prop:inequalities}
\end{proposition}
\begin{proof}
	The bipartite case was taken care of by point \ref{cor:bipartite-inequalities}
	in Corollary \ref{cor:bipartite-summary_2}.  Let $c=s_1\cdots s_n$ be any
	Coxeter element. Using elementary moves, (\ref{eqn:sums-in-sums}) and
	(\ref{eqn:F-invariance}) we get
	\begin{align*}
		F_{s_1cs_1}(\sigma_1(\alpha))+&F_{s_1cs_1}(\sigma_1(\gamma))
		=
		F_c(\alpha)+F_c(\gamma)\\
		&>\max\left\{F_c(\alpha+\gamma),F_c(\alpha\uplus_c\gamma)\right\}\\
		&=\max\left\{ F_{s_1cs_1}(\sigma_1(\alpha+\gamma)),F_{s_1cs_1}(\sigma_1(\alpha\uplus_c\gamma))\right\}\\
		&=\max\left\{ F_{s_1cs_1}(\sigma_1(\alpha)+\sigma_1(\gamma)),	F_{s_1cs_1}(\sigma_1(\alpha)\uplus_{s_1cs_1}\sigma_1(\gamma))\right\}
	\end{align*}
	as desired.

\end{proof}

\begin{proof}[Proof of Theorem \ref{thm-phi:polytopality}]
	It is enough to note that the Proposition \ref{prop:inequalities} together with Lemma
	\ref{lemma:reduction-to-c.e.} satisfy the requirements of Lemma
	\ref{lemma:criterion-polytopality}.

\end{proof}

\section{Relation between $\mathcal{F}_c^\Pi$ and the $c$-Cambrian fan $\mathcal{F}_c^C$}
\label{sec:cambrian}
We start by recalling some results and terminology from
\cite{sortable-polytopality}.

\begin{definition}
	(cf. Proposition 1.1 in \cite{sortable-polytopality}) Fix a Coxeter element
	$c$ and call an element $w\in W$ a \emph{$c$-singleton} if $w$ is both
	$c$-sortable and $c$-antisortable.
\end{definition}

Note that both $w_0$ and the identity element of $W$ are $c$-singletons for any
choice of $c$.  Denote by $\mathbf{w}_0$ the $c$-sorting word of $w_0$. 

\begin{theorem}
	(cf. Theorem 1.2 in \cite{sortable-polytopality}) An element $w\in W$ is a
	$c$-singleton if and only if it has a reduced expression which is a prefix of
	$\mathbf{w}_0$ up to commutations.
	\label{thm:singleton-is-prefix}
\end{theorem}

\begin{theorem}
	(cf. Theorem 2.6 in \cite{sortable-polytopality}) For any ray $\rho$ of
	$\mathcal{F}_c^C$, there exist a unique fundamental weight $\omega_i$ and a
	(non unique) $c$-singleton $w$ such that 
	$$
		\rho=\mathbb{R}_+\cdot w\omega_i
		\,.
	$$
	Conversely for any $c$-singleton $w$ and any fundamental weight $\omega_i$,
	the weight $w\omega_i$, lies on a ray of $\mathcal{F}_c^C$.
	\label{thm:rays-of-FC}
\end{theorem}

We will use Theorems \ref{thm:singleton-is-prefix} and \ref{thm:rays-of-FC} to
relate the rays of $\mathcal{F}_c^C$ to the elements of the set $\Pi(c)$.

\begin{definition}
	Given a Coxeter element $c\in W$, we call a reduced expression $c=s_1\cdots
	s_n$ \emph{greedy} if
  $$
		h(i,c)\ge h(j,c)
	$$
	whenever $i<j$.
\end{definition}

\begin{lemma}
	Any Coxeter element $c$ admits a greedy reduced expression.
\end{lemma}
\begin{proof}
	Consider a reduced expression $s_1\cdots s_n$ for $c$ and suppose $h(i,c)<
	h(i+1,c)$ for some $i$. Let $i$ be the minimal index with this property. Using
	Proposition 1.6 in \cite{shih} we can deduce that $i$ and $i+1$ are not
	connected in the Coxeter graph. Indeed if they were connected then we would
	have $i\prec_c i+1$ and thus $h(i,c)\ge h(i+1,c)$ which is in contradiction
	with our assumption. Therefore $s_i$ and $s_{i+1}$ commute and $s_1\cdots
	s_{i+1}s_i\cdots s_n$ is another reduced expression for $c$; we can now
	conclude by induction.

\end{proof}

\begin{remark}
	Greedy reduced expressions, in general, are not unique; for example
	$s_2s_4s_1s_3$ and $s_4s_2s_1s_3$ are both greedy reduced expression of the
	same Coxeter element in type $A_4$ (again we used the standard numeration of
	roots from \cite{bourbaki}).
\end{remark}

\begin{lemma}
	For any vertices $i$ and $j$ of the Dynkin diagram at distance $d$ from each
	other the difference  $h(i,c) - h(j,c)$ is at most $d$.
  \label{old-commuting-distance}
\end{lemma}
\begin{proof}
	It is enough to observe that if $i$ and $j$ are adjacent then either $i\prec_c
	j$ or $j\prec_c i$ so $|h(i,c) - h(j,c)|<1$ by Proposition 1.6 in \cite{shih}.
	Therefore each step on the minimal path in $I$  connecting $i$ and $j$
	contribute at most 1 to the difference $h(i,c) - h(j,c)$.

\end{proof}

Fix a greedy reduced expression for $c$. With some abuse of notation, we denote
this expression also as $c$.  Denote by $\mathbf{w}_m$ the sub-word of $c^m$
obtained by omitting in the $l$-th copy of $c$  all the transpositions $s_i$
such that $h(i,c)<l$. Observe that having taken a greedy reduced expression for
$c$, if we write $I_1, \dots I_m$ for the $c$-factorization of
$\mathbf{w}_m$, then
$$
	I_1\supseteq I_2\cdots\supseteq I_m.
$$
In particular if $\mathbf{w}_m$ is a reduced word then $w_m$, the corresponding element
of $W$, is $c$-sortable.
Let $m_c=\max_{i\in I}\{h(i,c)\}$, our goal is to show that the word
$\mathbf{w}_{m_c}$ is a reduced expression for $w_0$.

\begin{proposition}
	For any $i\in I$ and any $m\le h(i,c)$ we have $c^m\omega_i=\mathbf{w}_m\omega_i$.
	\label{prop:w-is-c}
\end{proposition}
\begin{proof}
	Let $I_1$, \dots $I_m$ be the $c$-factorization of $\mathbf{w}_m$ with respect to the
	fixed greedy reduced expression of $c$.  Observe that, for any $j$ appearing
	in  $I_{l+1}$ and for any  $k$ missing from $I_l$
  $$
		|h(k,c)- h(j,c)|\ge 2
	$$ 
	and so, by Lemma \ref{old-commuting-distance}, $s_k$ and $s_j$ commute.
	Consider now the element 
	$$
	\mathbf{w}=c_{I\setminus I_1}c_{I\setminus I_2}\cdots c_{I\setminus I_m}
	$$
	Since $m\le h(i,c)$, the reflection  $s_i$ will not appear in $\mathbf{w}$ and so
	$\mathbf{w}\omega_i=\omega_i$ hence $\mathbf{w}_m\mathbf{w}\omega_i=\mathbf{w}_m\omega_i$. Form the previous
	consideration we can move all the elements in the $l$-th copy of $c$ in
	$\mathbf{w}$
	up to the $l$-th block of $\mathbf{w}_m$ and obtain
	$$
		c^m\omega_i=
		c_{I_1}c_{I\setminus I_1}\cdots c_{I_m}c_{I\setminus I_m}\omega_i=
		\mathbf{w}_m\mathbf{w}\omega_i=
		\mathbf{w}_m\omega_i
		\,.
	$$

\end{proof}

\begin{proposition}
	$\mathbf{w}_{m_c}$ is a reduced expression of $w_0$.
\end{proposition}
\begin{proof}
	To show that $\mathbf{w}_{m_c}$ is an expression  of $w_0$ it is enough to
	show that both $w_0$ and $\mathbf{w}_{m_c}$ act in the same way on the weight
	space (the representation of $W$ as reflection group of $P_\mathbb{R}$ is
	faithful).  Fundamental weights form a basis of the weight space
	so it is enough to see how $w_0$ and $\mathbf{w}_{m_c}$  act on them.  For any $i$ we
	have $w_0\omega_i=-\omega_{i^*}$.  On the other hand, using Proposition
	\ref{prop:w-is-c}, we conclude that
  $$
		\mathbf{w}_{m_c}\omega_i=
		\mathbf{w}_{h(i,c)}\omega_i=
		c^{h(i,c)}\omega_i=
		-\omega_{i^*}
		\,.
	$$
	Therefore $\mathbf{w}_{m_c}$ is a word representing $w_0$.  The fact that
	it is a reduced
	expression follows from considerations on its length; each reflection $s_i$
	appears exactly $h(i,c)$ times in it.  Proposition 1.7 in \cite{shih} states
	that, for every $i$, the sum $h(i,c)+h(i^*,c)$ is equal to the Coxeter number
	$h$, hence
  $$
		\sum_{i\in I}\left( h(i,c)+h(i^*,c)\right)=|I|h=|\Phi|
	$$
  but in this way we are counting the contribution of each $i$ twice, i.e. 
  $$
		l(\mathbf{w}_{m_c})\le
		\sum_{i\in I}h(i,c)=
		\frac{1}{2}\sum_{i\in I}\left(h(i,c)+h(i^*,c)\right)=
		\frac{1}{2}|\Phi|=
		|\Phi_+|=
		l(w_0)
		\,.
	$$

\end{proof}

Note that, in view of last Proposition, for any $m\leq m_c$, $\mathbf{w}_m$ is a
reduced expression in $W$ (and $w_m$ is $c$-sortable).

\begin{proposition}
	Fix a greedy reduced expression for $c$. Then $\mathbf{w}_{m_c}$ is the
	lexicographically first reduced expression of $w_0$ as a sub-word of
	$c^\infty$. In other word $\mathbf{w}_{m_c}$ is the $c$-sorting word of $w_0$.
	\label{prop:w_m_w-is-c-sorting}
\end{proposition}
\begin{proof}
	It is enough to show that $w_m\alpha_i$ is a negative root for any $i$ not in
	$I_m$. We have
  $$
		0<
		(\alpha_i,\omega_i)=
		( w_m\alpha_i,w_m\omega_i)=
		( w_m\alpha_i,w_0\omega_i)=
  	( w_m\alpha_i,-\omega_{i^*})
	$$
	thus $( w_m\alpha_i,\omega_{i^*})<0$ and so $w_m\alpha_i$ is a negative root.

\end{proof}

\begin{remark}
	Combining together Theorem \ref{thm:singleton-is-prefix} and Proposition
	\ref{prop:w_m_w-is-c-sorting} we get another characterization of
	$c$-singletons: they are all the prefixes of $\mathbf{w}_{m_c}$ up to
	commutations.
\end{remark}

\begin{proposition}
	The sets of rays of $\mathcal{F}_c^\Pi$ and $\mathcal{F}_c^C$ coincide.
	\label{prop:equal-rays}
\end{proposition}
\begin{proof}
	Fix a greedy reduced expression for $c$.  Let $\rho$ be a ray of
	$\mathcal{F}_c^C$. By Theorem \ref{thm:rays-of-FC}  there exist a
	$c$-singleton $w$ and a fundamental weight $\omega_i$ such that
	$$
		\rho=\mathbb{R}_+w\omega_i.
	$$
	Let $m$ be the minimum integer such that $w$ is a prefix of $\mathbf{w}_m$.  By
	Proposition \ref{prop:w-is-c} 
	$$ 
	w\omega_i=\mathbf{w}_m\omega_i=c^m\omega_i\in\Pi(c).
	$$
	On the other hand, given any element $c^m\omega_i$ of $\Pi(c)$, let
	$\mathbf{w}_m$ be
	the corresponding sub-word of $c^m$ as in Proposition \ref{prop:w-is-c}; it is
	a $c$-singleton therefore $c^m\omega_i=\mathbf{w}_m\omega_i$ is a point on a ray of
	$\mathcal{F}_c^C$ by Theorem \ref{thm:rays-of-FC}.

\end{proof}

We can now define the polytope $\asso_c^a(W)$.  For any  point $a$ in
$P_\mathbb{R}$ and for any ray $\rho$ of $\mathcal{F}_c^C$ such that
$\rho=\mathbb{R}_+\cdot w\omega_j$, denote by $\mathcal{H}_\rho^a$ the
half-space
$$
	\mathcal{H}_\rho^a:=
	\left\{ \varphi\in P_\mathbb{R}^*
	\mid
	\varphi(w\omega_j)\le
	( a,\omega_j) \right\}
	\,.
$$
The main result in \cite{sortable-polytopality} is that, if $a$ lies in the
interior of the fundamental Weyl chamber,  the intersection of half-spaces

$$
	\asso_c^a(W):=
	\bigcap\mathcal{H}_\rho^a
$$
as $\rho$ runs over all rays of $\mathcal{F}_c^C$ is a simple polytope and its
normal fan is $\mathcal{F}_c^C$.

\begin{proof}[Proof of Theorem \ref{thm:polytopes-are-equal}]
	In view of Proposition \ref{prop:equal-rays}, the two polytopes become
	$$
		\asso_c^a(W)=
		\left\{ \varphi\in P_\mathbb{R}^* \mid 
		\varphi(c^m\omega_i)\le
		( a, \omega_i) 
		\, \forall i\in I , 0\le m\le h(i,c)\right\}	
	$$
  and
	$$
		\asso_c^f(W)=
		\left\{ \varphi\in P_\mathbb{R}^* \mid 
		\varphi(c^m\omega_i)\le
		f(i)
		\, \forall i\in I , 0\le m\le h(i,c)\right\}  
		\,.
	$$
	For any function $f:I\longrightarrow \mathbb{R}$ let $a$ be the point in
	$P_\mathbb{R}$ defined by the conditions
	$$
		(a,\omega_i):=f(i)
	$$
	for all $i\in I$.  Imposing condition \ref{thm-pi:polytopality:2} of Theorem
	\ref{thm-pi:polytopality} on $f$ is equivalent to ask for $a$ to lie in the
	fundamental Weyl chamber; indeed $a$ is in it if and only if the scalar
	product $( \alpha_j, a)$ is positive for every $j\in I$.  Since
	$\alpha_j=\sum_{i\in I}a_{ij}\omega_i$ we have
	$$
		( \alpha_j, a)=
		( \sum_{i\in I}a_{ij}\omega_i, a)=
		\sum_{i\in I}a_{ij}( a,\omega_i)=
		\sum_{i\in I}a_{ij}f(i)
		>0
		\,.
	$$
	We can thus conclude that, for any function $f:I\longrightarrow \mathbb{R}$
	satisfying conditions \ref{thm-pi:polytopality:1} and
	\ref{thm-pi:polytopality:2} of Theorem \ref{thm-pi:polytopality}, choosing
	$a$ as above, we get 
	$$
		\asso_c^a(W)=\asso_c^f(W)
		\,.
	$$

\end{proof}
\begin{remark}
	It is clear that, imposing condition \ref{thm-pi:polytopality:1} of Theorem
	\ref{thm-pi:polytopality}, from our construction we get only the polytopes from
	\cite{sortable-polytopality} obtained from points $a$ invariant under the
	action of $-w_0$. 
\end{remark}

\section*{Acknowledgments}
I would like to thank Andrei Zelevinsky for his guidance through the work on this
paper and for having introduced me to the theory of cluster algebras. Extremely
helpful was the suggestion to consider the negative roots $-\beta_i^c$ instead
of negative simples made by Hugh Thomas.  I am grateful to Sachin Gautam for the
tireless way he provided insights and the many useful conversations we had.  A
special thank to Nathan Reading for sharing his software and to Gregg Musiker
for explaining me the add-on to SAGE he is writing: both tools helped when
producing examples. I am also grateful to Elena Collina, Andrea Appel, Giorgia
Fortuna, Andrew Carrol, Federico Galetto, Shih-Wey Yang, Daniele Valeri, and
Alessandro D'Andrea for listening in moments I needed to express my doubts to
someone.

\nocite{sage}
\bibliographystyle{plain}
\bibliography{bibliography}

\end{document}